\documentclass [10pt, article, oneside, a4paper]{memoir}
\usepackage[T1]{fontenc}
\usepackage[utf8]{inputenc}

\usepackage[colorlinks]{hyperref}

\usepackage{hyphenat}
\usepackage{xcolor}
\usepackage[all]{xy}
\usepackage{amsfonts}
\RequirePackage{amsmath,amsthm,paralist,bm,amsthm,graphicx,color,mathrsfs,xcolor}
\xyoption{poly}

\newcommand{\aquatre}{%
\stockaiv\pageaiv%
\settypeblocksize{22cm}{13cm}{*}%
\setlrmargins{4cm}{*}{1}%
\setmarginnotes{0pt}{0pt}{0pt}%
\setulmargins{3.5cm}{*}{1}%
\setheadfoot{3\baselineskip}{3\baselineskip}%
\setheaderspaces{2\baselineskip}{*}{1}%
\checkandfixthelayout%
}
\aquatre

\newcommand{\addpoint}[1]{#1\ ---\ }

\setsecnumdepth{subsubsection}

\setsecnumformat{\csname the#1\endcsname---}
\setsubsubsechook{\setsecnumformat{{\normalfont\sffamily\csname
      the##1\endcsname\ }}}
\setsubsechook{\setsecnumformat{\csname the##1\endcsname\ ---\ }}
\setsechook{\setsecnumformat{\csname the##1\endcsname---}}

\setsecheadstyle{\centering\Large\bfseries\sffamily}
\setsubsecheadstyle{\large\bfseries\sffamily}
\setsubsubsecheadstyle{\bfseries\itshape\addpoint}
\setsubsubsecindent{1em}
\setbeforesubsubsecskip{.5em plus .2em minus -.1em}
\setaftersubsubsecskip{-0em}
\setsubparaheadstyle{\itshape}
\newtheoremstyle{thm}%
     {1.5ex plus .3ex minus .1ex}%
     {1ex plus .3ex minus .1ex}%
     {\itshape}%
     {}%
     {\sffamily}%
     {---}%
     {0em}%
     {$\bullet$\hbox{\ }#1\hbox{\ }#2}%

\theoremstyle{thm}

\newtheorem{definition}{Definition}[section]

\newtheorem{theorem}[definition]{Theorem}

\newtheorem{lemma}[definition]{Lemma}
\newtheorem{proposition}[definition]{Proposition}
\newtheorem{corollary}[definition]{Corollary}

\newtheoremstyle{note}%
     {1ex plus .3ex minus .1ex}%
     {1ex plus .3ex minus .1ex}%
     {}%
     {}%
     {\itshape}%
     {.}%
     {1em}%
     {}%
\theoremstyle{note}

\newtheorem{remark}[definition]{Remark}

\newlength{\remaining}
\newenvironment{intermediate}[1]
	{\par\medskip\noindent\makebox[2em][t]{\hfill#1\hfill}%
          \setlength{\remaining}{\textwidth}\addtolength{\remaining}{-2em}%
          \begin{minipage}[t]{\remaining}\itshape
          }
          {\end{minipage}\par\medskip}

\newcommand{\labelu}[1]{\POS[]\POS!U\drop{\raisebox{1em}{\strut$#1$}}}
\newcommand{\labelr}[1]{\POS[]\POS!R\drop{\strut\hspace{2ex}\rlap{$#1$}}}

\newcommand{\verticale}{\ar@{--}[d]}
\newcommand{\tr}[1]{*+[F]{#1}}

\newsavebox{\labelz}\newsavebox{\labelun}
\savebox{\labelz}{$\xymatrix@M=0.2ex{*++[F]{0}}$}
\savebox{\labelun}{$\xymatrix@M=0.2ex{*++[F]{1}}$}

\def\rebar#1{\expandafter\def\csname #1bar\endcsname{\overline{\csname
      #1\endcsname}}}		%

\newcommand{\M}{{\mathcal{M}}}
\newcommand{\FFF}{\mathfrak{F}}
\newcommand{\bbR}{\mathbb{R}}
\newcommand{\bbZ}{\mathbb{Z}}

\rebar{M}	%

\newcommand{\C}{\mathscr{C}}

\renewcommand{\SS}{\mathcal{S}}

\DeclareMathOperator{\Id}{Id}

\newcommand{\DSC}{{\text{\normalfont\sffamily DSC}}}
\newcommand{\ADSC}{\text{\normalfont\sffamily ADSC}}
\newcommand{\MCSC}{{\text{\normalfont\sffamily MCSC}}}

\newcommand{\up}[1]{\,\uparrow #1}

\newcommand{\card}{\#}

\newcommand{\Cstar}{\mathfrak{C}}

\newcommand{\slgb}{\mbox{$\sigma$-al}\-ge\-bra}

\newcommand{\BM}{\partial\M}
\newcommand{\Mtilde}{\widetilde{\M}}
\newcommand{\nutilde}{\widetilde{\nu}}
\newcommand{\Ttilde}{\widetilde{T}}

\newcommand{\vd}{\varepsilon}

\newcommand{\mutilde}{\widetilde{\mu}}
\newcommand{\Gtilde}{\widetilde{G}}
\newcommand{\rtilde}{\widetilde{r}}
\newcommand{\Sigmatilde}{\widetilde{\Sigma}}

\newcommand{\tq}{\;\big|\;}
\newcommand{\tqs}{\,\colon\,}

\newcommand{\ie}{\textsl{i.e.}}
\newcommand{\eg}{\textsl{e.g.}}

\newcommand{\un}{\mathbf{1}}

\newcommand{\height}{\tau}

\numberwithin{equation}{section}

\newcommand{\accessible}{accessible}
\newcommand{\Accessible}{Accessible}
\newcommand{\accessibility}{accessibility}

\begin{document}

\begin{center}
  {\huge\bfseries
  A spectral property for concurrent systems and some probabilistic applications}

\Large
\bigskip
Samy Abbes (\texttt{abbes@irif.fr})\\
{\normalsize Universit\'e de Paris --- IRIF (UMR 8243)}\\[1em]
\bigskip
Jean Mairesse (\texttt{jean.mairesse@lip6.fr})\\
{\normalsize CNRS --- LIP6 (UMR 7606)}\\[1em]
Yi-Ting Chen  (\texttt{yi-ting.chen@lip6.fr)}\\
{\normalsize Sorbonne Universit\'e --- LIP6 (UMR 7606)}\\
[1em]
\bigskip

December 2019--April 2020
\end{center}

\bigskip

\begin{abstract}
  We study trace theoretic concurrent systems. This setting encompasses safe (1-bounded) Petri nets. We introduce a notion of irreducible concurrent system and we prove the equivalence
 between irreducibility and a ``spectral property''. The spectral property states a strict inequality between radii of convergence of  certain growth series associated with the system. The proof that we present relies on Analytic combinatorics techniques. The spectral property is the cornerstone of our theory, in a framework where the Perron-Frobenius theory does not apply directly. This restriction is an inherent difficulty in the study of concurrent systems.

  We apply the spectral property to the probabilistic theory of concurrent systems. We prove on the one hand the uniqueness of the uniform measure, a question left open in a previous paper. 
 On the other hand,  we prove that this uniform measure can be realized as a Markov chain of states-and-cliques on a state space that can be precisely characterized. 
\end{abstract}

\section{Introduction}
\label{sec:introduction}

A trace monoid is a presented monoid of the form $\M=\langle\Sigma\;|\; ab=ba\text{ for $(a,b)\in I$}\rangle$ where $\Sigma$ is a finite alphabet and $I$ is an irreflexive and symmetric relation on~$\Sigma$.  From the concurrency theory viewpoint, elements of $\Sigma$ represent actions and pairs of actions $(a,b)\in I$ correspond to \emph{concurrent} actions. The concurrent actions of a pair $(a,b)\in I$ are such that the two successive occurrences $ab$ and $ba$ have the same effects, which corresponds to the identity $ab=ba$ in the quotient monoid~$\M$.

The elements of a trace monoid are called \emph{traces}. Trace monoids are already rich models of concurrency, used for instance for models of databases with parallel access~\cite{diekert90}. However trace monoids alone lack another feature essential for real life models: a notion of \emph{state}. Therefore, we enrich the trace monoid models by considering a monoid action $X\times \M\to X$ of a trace monoid $\M$ on a finite set of states~$X$. This yields a powerful model with two key ingredients: a build-in notion of concurrency \emph{and} a notion of state. How does the ``current state'' influence the future dynamics of the system? This is adequately rendered by introducing a particular sink state, which is to be avoided by all trajectories. In other words, an action leading to the sink state is actually a forbidden action, and this might depend on the current state. This setting encompasses popular concurrency models such as safe (a.k.a.,~$1$-bounded) Petri nets~\cite{desel04}.

\paragraph*{Dynamics of trace monoids and of concurrent systems.}
\label{sec:dynam-trace-mono}

Concurrent systems can be given a dynamics which naturally arises as a byproduct of their combinatorics~\cite{abbes15}. Hence, in order to equip concurrent systems with their ``natural'' dynamics, a first task is to extend to concurrent systems well known results on the combinatorics of trace monoids \cite{cartier69,viennot86,diekert90,diekert95}, a task already investigated in~\cite{abbes19:_markov}. We first recall the basic facts on how the combinatorics of a trace monoid induces a dynamics on traces, and then discuss their extension to concurrent systems.

Two objects collect most information on the combinatorics of traces. The first one is the M\"obius polynomial $\mu_\M(z)$ of the monoid, the definition of which is recalled in Section~\ref{sec:trace-monoids}. The second one is the \emph{digraph of cliques} of the monoid. Traces admit a normal form, similar to the Garside normal form for braids. Normal forms of traces correspond bijectively to the paths of the digraph of cliques. The trace monoid is \emph{irreducible} if the graph $\bigl(\Sigma,(\Sigma\times\Sigma)\setminus I\bigr)$ is connected, and in this case, the digraph of cliques is strongly connected and aperiodic. The \emph{growth rate} of the monoid, defined as the exponential rate of its growth series, is given by the inverse~$r^{-1}$, where $r$ is the unique root of smallest modulus of the M\"obius polynomial.


The compactification of the trace monoid $\M$ by its \emph{boundary at infinity} $\BM$ provides a topological space carrying a unique uniform measure, which extends the uniform Bernoulli measure on infinite sequences of letters~\cite{abbes15}. The characterisation of this measure involves again the root~$r$. Furthermore, this uniform measure can be realized by a Markov chain on the digraph of cliques whose parameters admit a combinatorial description. 

The extension to concurrent systems of the different elements introduced above for trace monoids goes as follows. The M\"obius polynomial has to be replaced by a matrix polynomial, the \emph{M\"obius matrix}. The root $r$ becomes the root of smallest modulus of the determinant of the M\"obius matrix. 
The uniform measure can still be realized by a Markov chain but the digraph of cliques has to be replaced by the digraph of \emph{states-and-cliques} (\DSC), 
 which describes the combinatorics of traces \emph{and}  the action of the monoid. Quite surprinsingly, the \DSC\ has no reason to be strongly connected. Whereas, a cornerstone of trace monoids theory is that the digraph of cliques of an irreducible trace monoid is strongly connected and aperiodic.

Hence, the notion of irreducibility for concurrent systems is one of the keys for understanding their dynamics. And, once a correct notion of irreducibility is identified, one has to accept the fact that it does not imply the strong connectedness of the \DSC---which, incidently, prevents from a straightforward application of Perron-Frobenius theory. The interplay between an ``irreducibility'' of a concurrent system, the structure of its \DSC, and the dynamics it induces, is thus more involved than for trace monoids alone and deserves a detailled study.

\paragraph*{Targeted applications.}
\label{sec:targ-appl}

Different research fields may have an interest in a theory of probabilistic concurrent systems, and for different reasons. The probabilistic model checking of concurrent systems lacks such a theory. Instead, computer scientists rely on \emph{sequential} probabilistic models including nondeterminism to take into account concurrency~\cite{baier08,kwiatkowska12}, which are not well suited for all situations. Hence a trace theory of probabilistic concurrent systems fills a gap in discrete events systems theory with applications in probabilistic simulation and in probabilistic model checking for real life models. The detailed study of the uniform measure is of particular interest in this respect.
%
%

\paragraph*{Description of the results.}
\label{sec:description-results}

As mentioned above, a first milestone is to find an adequate notion of irreducibility for concurrent systems. Our notion of irreducibility is introduced  in Section~\ref{sec:concurrent-systems}, it asks for the graph of the action to be strongly connected and for the trace monoid to be irreducible. The spectral property, that we now describe, confirms that this is indeed an interesting definition. 
A concurrent system, resulting from the action of a trace monoid $\M(\Sigma,I)$ on a set of states~$X$, has the ``spectral property'' if: for any action $a\in\Sigma$, the new concurrent system obtained by restriction after removing the action $a$ is ``significantly smaller'' than the original concurrent system. The latter notion is rigorously formalized through the notion of radius of convergence of the growth series of the system. We show in Theorem~\ref{thr:1} that irreducible concurrent systems have the spectral property.


This spectral property is, of course, also true for irreducible trace monoids. One way to prove it for trace monoids is to use the Perron-Frobenius theory for primitive matrices.
But the straightforward application of the Perron-Frobenius theory is unavailable for concurrent systems since their \DSC\ is not strongly connected in general. Instead, we rely on other techniques, namely on an \emph{ad hoc} construction with the help of elementary, yet powerful tools from Analytic combinatorics. 

Based on the spectral property that we prove for irreducible concurrent systems, we are able to further investigate the structure of the $\DSC$. In particular, we introduce the notion of positive and of null node of the $\DSC$, and we show that null nodes can be safely removed without affecting, asymptotically, the combinatorics of the system. The existence of null nodes is both a difficulty and a specificity of concurrent systems; they are absent from irreducible trace monoids theory.

The combinatorics results that we obtain have natural applications in the theory of probabilistic concurrent systems. We develop the probabilistic material previously introduced in~\cite{abbes19:_markov}, where a notion of uniform measure for concurrent systems was constructed, in the light of our new results. In particular, we prove the uniqueness of the uniform measure, a result which was left open in \cite{abbes19:_markov}. Furthermore, 
we show that the associated Markov chain of states-and-cliques only visits ``positive nodes'' of the $\DSC$. This result provides a natural probabilistic interpretation of our combinatorial results. 

\paragraph*{Uniform measure and uniform measures.}
\label{sec:unif-meas-unif}

The definition of uniform measure adopted in this paper may seem arbitrary: it is defined as a probability measure on the boundary at infinity of the trace modnoid of the form $\mu_\alpha(x)=s^{|x|}\Gamma(\alpha,\alpha\cdot x)$ for some real $s>0$ and some positive \emph{cocycle}~$\Gamma$, \ie, a positive function $\Gamma: X\times X\to\bbR_{>0}$ satisfying $\Gamma(\alpha,\gamma)=\Gamma(\alpha,\beta)\Gamma(\beta,\gamma)$ for all states~$\alpha,\beta,\gamma$ (see Section~\ref{sec:appl-spectr-prop}).

Other definitions might be considered. For instance, one could consider for each non negative integer~$k\geq0$ and for each initial state~$\alpha\in X$, the set of executions of length~$k$ starting from~$\alpha$, and the uniform distribution $\mu_{\alpha,k}$ on this finite set; then consider the weak limit of the family~$(\mu_{\alpha,k})_{k\geq0}$. Such a limit, if it exists, is a legitimate candidate for a notion of \emph{uniform meaure at infinity}.

It turns out that both notions actually coincide. The proof of this result involves a general study of the uniform distributions on paths in general finite graphs, a topic outside the scope of this paper. The formulation of convergence results in the framework of general graphs is the topic of a separate paper currently in preparation. In the present paper, a glimpse to this aspect of the theory is introduced in Section~\ref{sec:uniq-unif-meas}, with a reference to a result from the literature on reducible non negative matrices.

\section{Preliminaries}
\label{sec:preliminaries}

We shall use throughout the paper the following notions and elementary properties. A \emph{digraph} is a pair $(N,E)$ where $N$ is a finite set, called the set of \emph{nodes}, and~$E$, the set of \emph{arcs}, is a subset of $N\times N$. A \emph{graph} is a digraph $(N,E)$ where $E$ is symmetric, \ie, $(x,y)\in E\implies (y,x)\in E$. For brevity, we sometimes denote graphs and digraphs by their sets of nodes, the sets of arcs being understood and conventionally denoted by~$E$.

A \emph{path} in a digraph $N$ is a sequence $(x_1,\dots,x_p)$ of nodes, maybe empty, such that $(x_i,x_{i+1})\in E$ for $i=1,\dots,p-1$. The \emph{length} of the path is the integer~$p$. In particular, single nodes correspond to paths of length~$1$, and the empty path is the unique path of length~$0$. We say that the path $(x_1,\dots,x_p)$ \emph{leads from $x_1$ to~$x_p$}. If $x$ and $y$ are two nodes, we write $x\to^* y$ if there exists a path leading from $x$ to~$y$.

A digraph $N$ is \emph{strongly connected} if $x\rightarrow^* y$ for every two nodes $x$ and~$y$. If $N$ is a digraph and if $N'$ is a subset of~$N$, the digraph \emph{induced by~$N'$} is $\bigl(N',E\cap(N'\times N')\bigr)$. We simply denote it by~$N'$. A \emph{subgraph} of $N$ is any digraph induced by some subset of~$N$. A \emph{strongly connected component} of $N$ is any maximal strongly connected subgraph of~$N$, where the maximality is understood with respect to inclusion. If $N_1,\dots, N_p$ are the strongly connected components of~$N$, then the collection of sets $\{N_1,\dots,N_p\}$ is a partition of the set~$N$.

Let $S=\{N_1,\dots,N_p\}$ be the collection of strongly connected components of a digraph~$N$. Then for all $i,j\in\{1,\dots,p\}$ and for all nodes $x,x'\in N_i$ and $y,y'\in N_j$, one has: $x\to^* y\iff x'\to^*y'$. We write $N_i\preceq N_j$ if there are nodes $x\in N_i$ and $y\in N_j$ such that $x\to^*y$. This relation is a partial order on~$S$. We say that $N_j$ is \emph{terminal} if $N_j$ is maximal in~$(S,\preceq)$. Since we assume $N$ to be finite, $S$~is itself finite and non empty if $N\neq\emptyset$. Therefore, in this paper, non empty digraphs always have at least one terminal strongly connected component.

The \emph{spectral radius} $\rho(A)$ of a real or complex matrix $A$ is the maximal modulus of its complex eigenvalues. A \emph{nonnegative matrix} is a real matrix, the entries of which are all non negative. We write $A\leq B$ for two matrices of the same size if $B-A$ is nonnegative.

Let $(N,E)$ be a digraph. Its \emph{adjacency matrix} is the $\{0,1\}$-matrix $F$ indexed by~$N$, and with $F_{x,y}=\un\bigl((x,y)\in E\bigr)$, where, for some predicate~$P$, we denote by $\un(P)$ the characteristic function of~$P$. For $z$ a node, let $\un_z$ denote the vector indexed by~$N$ and defined by $\un_z(x)=\un(x=z)$, and let $\un'_z$ denote its transpose vector. Then, for every integer $n\geq1$, and for every two nodes $x$ and~$y$, the number $\lambda_{x,y}(n)$ of paths of length $n$ leading from $x$ to $y$ is given by:
\begin{gather}
  \label{eq:15}
  \lambda_{x,y}(n)=\un'_x F^{n-1}\un_y\,,
\end{gather}
after identification of the above $1\times1$ matrix with its unique entry.

The \emph{spectral radius} $\rho(N)$ of a digraph $N$ is the spectral radius of its adjacency matrix. For each pair $(x,y)\in N\times N$, let $G_{x,y}(z)$ be the series:
\begin{gather*}
  G_{x,y}(z)=\sum_{n\geq0}\lambda_{x,y}(n)z^n,
\end{gather*}
of radius of convergence $r_{x,y}\in(0,+\infty]$. We define:
\begin{gather*}
r(N)=\min\{r_{x,y}\tqs(x,y)\in N\times N\}  
\end{gather*}
which is a positive number, maybe equal to~$+\infty$. Alternatively, $r$ is the radius of convergence of the matrix series $\sum F^nz^n$.

Next proposition is a folklore result. The first part is the classical result which relates the spectral radius of an operator $A$ with the radius of convergence of the series $\sum A^nz^n$. The second part is based on the strong results of the Perron-Frobenius theory for irreducible matrices, see for instance~\cite{seneta81}. 

\begin{proposition}
  \label{prop:6}
  Let $N$ be a digraph, and let $r=r(N)$ and $\rho=\rho(N)$. Then $r=\rho^{-1}$ if $\rho\neq0$ and $r=+\infty$ if $\rho=0$. Furthermore, there exists a node $x$ and a positive integer $d$ such that, for any positive multiple $d'$ of~$d$, the series $\sum_{n\geq0}\lambda_{x,x}(nd')z^n$ has $r^{d'}$ as radius of convergence.
\end{proposition}

\begin{remark}
  \label{rem:2}
  If $G(z)=\sum_{n\geq0}a_nz^n$ is a rational series with radius of convergence $r<\infty$, and if all the coefficients $a_n$ are non negative, then it is well known that $r$ is a pole of~$G(z)$. Indeed, it follows from Pringsheim's theorem \cite[Th.~IV.6]{flageolet09} that $r$ is a singularity of~$G(z)$, and for rational series, poles and singularities coincide.
\end{remark}

\section{Trace monoids and concurrent systems}
\label{sec:trace-mono-conc}

\subsection{Trace monoids}
\label{sec:trace-monoids}

A \emph{trace monoid} is a finitely generated monoid of the form
$$
\M=\langle\,\Sigma\tq ab=ba\quad\forall (a,b)\in I\rangle,
$$
where $\Sigma$ is a finite alphabet, and $I\subseteq\Sigma\times\Sigma$ is an irreflexive and symmetric relation on~$\Sigma$. We denote it by $\M=\M(\Sigma,I)$. Elements of $\Sigma$ are called \emph{letters} and elements of $\M$ are called \emph{traces}. We identify letters and their image in $\M$ through the canonical surjection $\Sigma^*\to\M$. More generally, according to the context, it should be clear whether a word $a_1\dots a_p$ represents an element of the free monoid $\Sigma^*$ or its image in~$\M$. The empty word and the unit element of $\M$ are both denoted~$\vd$. Elements of $\M$ distinct from $\vd$ are called \emph{non empty traces}. A trace is thus a congruence class of the free monoid~$\Sigma^*$, relatively to the smallest congruence containing all pairs of the form $(ab,ba)$, with $(a,b)\in I$. The \emph{length} of a trace~$x$, denoted by~$|x|$, is the length of any word representing~$x$. If $H$ is a subset of~$\Sigma$, we denote by $\langle H\rangle$ the submonoid of $\M$ generated by~$H$. Obviously: $\langle H\rangle=\M(H,I_H)$ with $I_H=(H\times H)\cap I$.

The alphabet $\Sigma$ is thought of as the alphabet of elementary actions of a system, some of them being concurrent. The concurrency of actions is encoded into the relation~$I$, which is called \emph{the independence relation}. Two actions $a$ and $b$ are concurrent if and only if $(a,b)\in I$, which translates as $ab=ba$ in the monoid. This identification characterizes the so-called \emph{trace models} or \emph{partial order models} of concurrency~\cite{diekert95}.

The \emph{dependence relation} on $\Sigma$ is the relation $D=(\Sigma\times\Sigma)\setminus I$. The \emph{Coxeter graph} of $\M$ is the graph $(\Sigma,D)$.  The monoid $\M$ is \emph{irreducible} if the graph $(\Sigma,D)$ is connected. Equivalently, $\M$~is irreducible if and only if, for every subsets $\Sigma_1,\Sigma_2\subseteq\Sigma$, if $\M$ is isomorphic to the direct product $\langle\Sigma_1\rangle\times\langle\Sigma_2\rangle$ then $\Sigma_1=\emptyset$ or $\Sigma_2=\emptyset$.

Much of the combinatorics results on trace monoids are based on the existence of a normal form for traces, which we describe now. A \emph{clique} of the monoid is a trace which can be written as a product of pairwise commuting letters, hence $c=a_1\ldots a_p$ with $(a_i,a_j)\in I$ for $i\neq j$.  This writing is unique up to the order of occurrences of the letters. Let $\C$ denote the set of cliques, and let $\Cstar=\C\setminus\{\vd\}$ be the set of non empty cliques. A pair $(c,c')$ of non empty cliques, where $c=a_1\cdot\ldots\cdot a_p$ and $c'=b_1\cdot\ldots\cdot b_q$ with $a_1,\dots,a_p,b_1,\dots,b_q\in\Sigma$, is said to be in \emph{normal form} if for every letter $b_j$ there exists a letter $a_i$ such that $(a_i,b_j)\notin I$. We denote this relation by $c\to c'$. A \emph{normal sequence} is a sequence $(c_1,\dots,c_h)$ of non empty cliques such that $c_i\to c_{i+1}$ holds for $i=1,\dots,h-1$. For every  non empty trace~$x$, there exists a unique positive integer $h$ and a unique normal sequence $(c_1,\dots,c_h)$ such that $x=c_1\cdot\ldots\cdot c_h$~\cite{cartier69}, \cite[Ch.~11]{lallement79}. This normal sequence is the \emph{(Cartier-Foata) normal form} of~$x$, and the integer $h$ is the \emph{height} of~$x$, which we denote by $h=\height(x)$. By convention, the normal form of the empty trace is the empty sequence, and $\height(\vd)=0$.

Consider the digraph with non empty cliques as nodes, and an arc from $c$ to $c'$ whenever $c\to c'$ holds. This is the \emph{digraph of cliques} of the monoid. Paths in this digraph correspond to normal sequences of non empty cliques. Their set is thus in bijection with~$\M$. If $\M$ is irreducible, then this digraph is strongly connected and aperiodic~\cite{krob03,abbes15}.

The growth series $G(z)$ of the trace monoid~$\M$, and the M\"obius polynomial $\mu(z)$ of~$\M$, are defined by:
\begin{align*}
  G(z)&=\sum_{n\geq0}\card\M(n)z^n \quad\text{with $\M(n)=\{x\in\M\tqs|x|=n\}$,}\\
\mu(z)&=\sum_{c\in\C}(-1)^{|c|}z^{|c|}.
\end{align*}

The series $G(z)$ is the formal inverse of~$\mu(z)$: $\mu(z)G(z)=1$ \cite{cartier69,viennot86}. In particular, it is a rational series.

The following results can be obtained as consequences of the above, combined with the Perron-Frobenius theory for primitive matrices~\cite{seneta81}. Assume that $\Sigma\neq\emptyset$. Then, among the complex roots of~$\mu(z)$, only one has smallest modulus~\cite{goldwurm00,krob03}. This root is real positive and lies in $(0,1]$. It is a simple root of $\mu(z)$ if $\M$ is irreducible~\cite{krob03}. This root is also the radius of convergence of~$G(z)$.

\subsection{Basics of concurrent systems}
\label{sec:concurrent-systems}

Recall that a monoid $M$ with unit element $\vd$ is said to \emph{act on the right} on a set $Y$ if there is a mapping $Y\times M\to Y$, denoted $(\alpha,y)\mapsto \alpha\cdot y$, satisfying the two following properties:
\begin{align*}
  \forall \alpha \in Y\quad\alpha\cdot\vd&=\alpha,&\text{and\quad}
                                                    \forall \alpha\in Y\quad\forall y,z\in M\quad(\alpha\cdot y)\cdot z=\alpha\cdot(yz).
\end{align*}

\begin{definition}
  \label{def:1}
  A \emph{concurrent system} is a triple $(\M,X,\bot)$, where $X$ is a finite set, $\bot$~is a distinguished symbol not in~$X$, and $\M$ is a trace monoid acting on $X\cup\{\bot\}$, such that $\bot\cdot x=\bot$ for all $x\in\M$.

  Elements of $X$ are called \emph{states}. A trace $x\in\M$ such that $\alpha\cdot x\neq\bot$ for some state~$\alpha$, is an \emph{execution of the system from~$\alpha$}, or simply an \emph{execution} if $\alpha$ is understood. The execution $x$ is said to \emph{lead} from $\alpha$ to the state~$\alpha\cdot x$.
  
  The concurrent system $(\M,X,\bot)$ is:
  \begin{itemize}
  \item \emph{Trivial} if:\quad$\forall(\alpha,x)\in X\times\M\quad \alpha\cdot x=\bot$; it is \emph{non trivial} otherwise.
  \item \emph{\Accessible} if:\quad $\forall (\alpha,\beta)\in X\times X\quad\exists x\in\M\quad \alpha\cdot x=\beta$.
  \item \emph{Alive} if:\quad for every state $\alpha$ and for every letter $a$ in the base alphabet of~$\M$, there exists an element $x\in\M$ with at least one occurrence of $a$ and such that $\alpha\cdot x\neq\bot$.
  \item \emph{Irreducible} if:\quad it is \accessible{} and alive and $\M$ is an irreducible trace monoid.
  \end{itemize}
\end{definition}

With each trace monoid, we canonically associate a concurrent system as follows.

\begin{definition}
  \label{def:4}
Let $\M$ be a trace monoid. Pick two distinct symbols $*$ and~$\bot$, and consider the unique monoid action of $\M$ on the singleton $X=\{*\}$, extended to $\bot\cdot x=\bot$ for all $x\in\M$. The concurrent system $\Mbar=(\M,X,\bot)$ thus defined is \emph{canonically associated with~$\M$}.
\end{definition}

The concurrent system $\Mbar$ associated with a trace monoid $\M$ is trivially \accessible{} and alive. Hence $\Mbar$ is irreducible as a concurrent system if and only if $\M$ is irreducible as a trace monoid.

\subsection{Combinatorics of concurrent systems}
\label{sec:comb-conc-syst}

\subsubsection{Definitions and notations}
\label{sec:defin-notat}

Given a concurrent system $(\M,X,\bot)$, we will use the following notations, defined for $\alpha,\beta$ ranging over~$X$, $n$~ranging over the non negative integers and where $z$ is a formal variable~:
  \begin{align*}
    \M_\alpha&=\{x\in\M\tqs \alpha\cdot x\neq\bot\}&\M_{\alpha,\beta}&=\{x\in\M_\alpha\tqs \alpha\cdot x=\beta\}\\
    \M_\alpha(n)&=\{x\in\M_\alpha\tqs|x|=n\}&    \M_{\alpha,\beta}(n)&=\{x\in\M_{\alpha,\beta}\tqs|x|=n\}\\
  G_{\alpha}(z)&=\sum_{n\geq0}\card\M_{\alpha}(n)z^n&
G_{\alpha,\beta}(z)&=\sum_{n\geq0}\card\M_{\alpha,\beta}(n)z^n\\
\C_{\alpha}&=\M_{\alpha}\cap\C&
\Cstar_\alpha&=\C_\alpha\setminus\{\vd\}\\
    \Sigma_\alpha&=\M_\alpha\cap\Sigma& \C_{\alpha,\beta}&=\M_{\alpha,\beta}\cap\C\\
\mu_{\alpha,\beta}(z)&=\sum_{c\in\C_{\alpha,\beta}}(-1)^{|c|}z^{|c|}
  \end{align*}

\begin{definition}
  \label{def:3}
Let $(\M,X,\bot)$ be a concurrent system. For each pair $(\alpha,\beta)\in X\times X$, we denote by $r_{\alpha,\beta}$ the radius of convergence of the series $G_{\alpha,\beta}(z)$, maybe equal to~$+\infty$. The number $r=\min\{r_{\alpha,\beta}\tqs(\alpha,\beta)\in X\times X\}$, maybe equal to~$+\infty$, is called the \emph{characteristic root} of the system.

  The matrix $\mu(z)=\bigl(\mu_{\alpha,\beta}(z)\bigr)_{(\alpha,\beta)\in X\times X}$ is the \emph{M\"obius matrix} of the concurrent system; the matrix $G(z)=\bigl(G_{\alpha,\beta}(z)\bigr)_{(\alpha,\beta)\in X\times X}$ is its \emph{growth matrix}.
\end{definition}

As an extension of the classical result for trace monoids recalled earlier, we have that the M\"obius matrix is the formal inverse of the growth matrix~\cite[Th.~5.10]{abbes19:_markov}:
\begin{gather}
  \label{eq:2}
\mu(z)G(z)=G(z)\mu(z)=\Id.  
\end{gather}
This formal identity translates as an identity in the algebra of real matrices if $z\in[0,r)$, where $r$ is the characteristic root of the system.

\begin{remark}
  \label{rem:1}
Let $(\M,X,\bot)$ be a concurrent system with $\M=\M(\Sigma,I)$.  Let $\Sigma=\Sigma^1+\Sigma^2$ be a partition of $\Sigma$ by two subsets such that $\Sigma^1\times\Sigma^2\subseteq I$; hence, any two actions of $\Sigma^1$ and of $\Sigma^2$ are concurrent. Let $\M^1=\langle\Sigma^1\rangle$ and $\M^2=\langle\Sigma^2\rangle$, and let $\mu^1(z)$ and $\mu^2(z)$ be the M\"obius matrices of the concurrent systems $(\M^1,X,\bot)$ and $(\M^2,X,\bot)$, where the action of $\M^i$ on $X\cup\{\bot\}$ is obtained by restriction of the action of~$\M$ on $X\cup\{\bot\}$ for $i=1,2$. Then for any two states $\alpha,\beta\in X$, the set $\C_{\alpha,\beta}$ can be written as follows, where ``$\sum$'' stands for the disjoint union: $\C_{\alpha,\beta}=\sum_{\gamma\in X}\bigl(\C_{\alpha,\gamma}^1\times\C_{\gamma,\beta}^2\bigr)$, where $\C_{\alpha,\gamma}^i=\C_{\alpha,\gamma}\cap\M^i$ for $i=1,2$. In turn, one obtains a product form for the M\"obius matrix: $\mu(z)=\mu^1(z)\mu^2(z)=\mu^2(z)\mu^1(z)$; in particular, $\mu^1(z)$~and $\mu^2(z)$ commute.
\end{remark}

The following result is elementary.

\begin{proposition}
  \label{prop:3}
  Let $(\M,X,\bot)$ be a non trivial and accessible concurrent system. Then its characteristic root $r$ satisfies $r<\infty$.
\end{proposition}

\begin{proof}
  Since the system is non trivial, there exists a state $\alpha\in X$ and a letter $a$ such that $\alpha\cdot a\neq\bot$. Let $\beta=\alpha\cdot a$ and, since the system is accessible, let $x$ be an execution leading from $\beta$ to~$\alpha$. Put $p=|a x|$. Then $\M_{\alpha,\alpha}(kp)\geq1$ for every $k\geq0$, hence $r_{\alpha,\alpha}\leq1$ and thus $r<\infty$.
\end{proof}



The following result is proved in~\cite{abbes19:_markov}. We give a proof here for the sake of completeness.

\begin{proposition}
  \label{prop:4}
  Let $(\M,X,\bot)$ be a non trivial and accessible concurrent system of characteristic root~$r$. Then all growth series $G_{\alpha,\beta}(z)$ are rational series with same radius of convergence~$r$, and $r$ is a root of smallest modulus of the polynomial $\theta(z)=\det\mu(z)$. 
\end{proposition}

\begin{proof}
  The inversion formula $G(z)\mu(z)=\Id$ recalled in~(\ref{eq:2}) shows, \emph{via} the determinant formula for the inverse of a matrix, that all growth series $G_{\alpha,\beta}(z)$ are rational.

  Since the growth series $G_{\alpha,\beta}(z)$ has non negative coefficients,  its radius of convergence $r_{\alpha,\beta}$ is one of its poles (see Remark~\ref{rem:2}). Let $(\alpha,\beta)$ and $(\alpha',\beta')$ be two pairs of states. Let also $x$ and $y$ be executions such that $\alpha\cdot x=\alpha'$ and $\beta'\cdot y=\beta$, and put $p=|x|$ and $q=|y|$. Then $\M_{\alpha,\beta}(n+p+q)\supseteq \{xuy\tqs u\in\M_{\alpha',\beta'}(n)\}$ for every integer $n\geq0$. Since trace monoids are left and right cancellative, it implies:
  \begin{gather*}
    \card\M_{\alpha,\beta}(n+p+q)\geq \card\M_{\alpha',\beta'}(n),
  \end{gather*}
and thus for all reals $z$ where the right-hand series converges:
\begin{gather}
  \label{eq:16}
G_{\alpha',\beta'}(z)\leq \frac1{z^{p+q}}  G_{\alpha,\beta}(z).
\end{gather}
Hence $r_{\alpha',\beta'}\geq r_{\alpha,\beta}$. Inverting the roles of $(\alpha,\beta)$ and of $(\alpha',\beta')$, we also obtain the converse inequality, and finally $r_{\alpha,\beta}=r_{\alpha',\beta'}=r$.

The formula $G(z)\mu(z)=\Id$, valid in the algebra of complex matrices for all complexes $z$ with $|z|<r$, shows that all root of $\theta(z)$ have their modulus at least equal to~$r$. It is thus enough, to complete the proof, to show that $\theta(r)=0$. Seeking a contradiction, assume that $\theta(r)\neq0$, and thus that $\mu(r)$ is invertible. It implies that $G(z)$ is bounded on $[0,r)$, which contradicts that $r$ is a pole of all the series~$G_{\alpha,\beta}(z)$.
%
\end{proof}

If $\M$ is a trace monoid, the M\"obius matrix of the concurrent system $\Mbar$ associated with $\M$ as in Definition~\ref{def:4} is the $1\times1$ matrix with the M\"obius polynomial $\mu_\M(z)$ of $\M$ as unique entry. The characteristic root of $\Mbar$ is thus the unique root of smallest modulus of~$\mu_\M(z)$ (see Section~\ref{sec:trace-monoids}).

\subsubsection{Graph of states of a concurrent system}
\label{sec:digr-stat-cliq}

The first natural object to consider is the ``graph of states'' of a concurrent system, defined as follows. It is actually a labeled multigraph. However this object will be of little use for us, except for graphical representations purposes. Therefore we do not insist on the adapted notion of multigraph.

\begin{definition}
\label{def:9}
  Let $(\M,X,\bot)$ be a concurrent system with $\M=\M(\Sigma,I)$. The \emph{multigraph of states} of the concurrent system is the labeled multigraph $(X,E)$, with a node for each state of the concurrent system, and where there is an arc labeled by $a\in\Sigma$ from a state $\alpha$ to a state $\beta$ whenever $\alpha\cdot a=\beta$.
\end{definition}

In particular, the multigraph of states must have a diamond shape for at least every pair $(a,b)\in I$ and every state $\alpha$ such that $\alpha\cdot (ab)\neq\bot$, since then $(\alpha\cdot a)\cdot b$ and $(\alpha\cdot b)\cdot a$ correspond to two paths in the multigraph with the same origin and the same destination. See illustrations in Section~\ref{sec:examples}.

\subsubsection{Digraphs of states-and-cliques and its augmented version (\/$\DSC$\ and\/ $\ADSC$)}
\label{sec:digr-stat-cliq-1}

For combinatorics purposes, and for instance for counting the executions of a concurrent system, the multigraph of states is of little help. Indeed, two different paths in the multigraph of states, of the form $(\alpha\cdot a)\cdot b$ and $(\alpha\cdot b)\cdot a$ with $ab=ba$, count for only one execution.

Instead, one must rely on the normal form of traces, and thus of executions. For this purpose, we introduce two digraphs related to the normal form of traces, adapted to the framework of concurrent systems. The \emph{digraph of states-and-cliques} (\DSC) is the analogous of the digraph of cliques for trace monoids; the extension to the $\ADSC$ in the following definition mimics the analogous introduced for trace monoids in~\cite{krob03}.

\begin{definition}
  \label{def:5}
  Let $\SS=(\M,X,\bot)$ be a concurrent system. The \emph{digraph of states-and-cliques (\DSC)} of $\SS$ is the directed graph described as follows: its nodes are all pairs $(\alpha,c)$ such that $\alpha$ is a state and $c\in\Cstar_\alpha$; and there is an arc from a node $(\alpha,c)$ to a node $(\beta,d)$ if and only if $\beta=\alpha\cdot c$ and if the relation $c\to d$ holds.

  The \emph{augmented digraph of states-and-cliques (\ADSC)} of $\SS$ has a node for each triple of the form $(\alpha,c,i)$, where $(\alpha,c)$ is a node of the \DSC, and $i=1,\dots,|c|$. There is an arc from $(\alpha,c,i)$ to $(\beta,d,j)$ if:
\begin{enumerate}
\item $(\alpha,c)=(\beta,d)$ and $j=i+1$; or
\item $i=|c|$ and $j=1$ and there is an arc from $(\alpha,c)$ to $(\beta,d)$ in the \DSC.
\end{enumerate}
\end{definition}

\begin{proposition}
  \label{prop:5}
  Let $\SS=(\M,X,\bot)$ be a concurrent system. Then there are 1-1 correspondences between:
  \begin{itemize}
  \item Executions of\/ $\SS$ of height $h$ on the one hand, and paths of length $h$ in the \DSC\ on the other hand, for every integer $h\geq1$.
  \item Executions of\/ $\SS$ of length $n$ on the one hand, and paths of length $n$ in the \ADSC\ which lead from a node of the form $(\alpha,c,1)$ to a node of the form $(\beta,d,|d|)$ on the other hand, for every integer $n\geq1$.
  \item Executions of\/ $\SS$ leading from a state $\alpha$ to a state $\beta$ and of length $n$ on the one hand, and paths of length $n$ in $\ADSC$, leading from a node of the form $(\alpha,c,1)$ to a node of the form $(\beta,d,|d|)$ on the other hand, for every integer $n\geq1$.
  \end{itemize}
\end{proposition}

\begin{proof}
  Consider an initial state~$\alpha$, and an execution $x\in\M_\alpha$\,, $x\neq\vd$. Then~$x$, as a trace in~$\M$, has a unique normal form $(c_1,\dots,c_{h})$ where $h$ is the height of~$x$. Put $\alpha_0=\alpha$, and define by induction $\alpha_{i+1}=\alpha_i\cdot c_{i+1}$ for $i=0,\dots,h-1$. Then the sequence $(\alpha_i,c_{i+1})_{0\leq i<h}$ is a path of length $h$ in the \DSC. The uniqueness of the normal form of traces implies that the correspondence between executions of height $h$ and paths of length $h$ is the \DSC\ is a bijection.

  To each node $(\alpha,c)$ in the \DSC\ is associated the mandatory path
  $$
(\alpha,c,1),(\alpha,c,2),\dots,(\alpha,c,|c|)
  $$
of length $|c|$ in the \ADSC. Hence the previous correspondence gives rise to a correspondence between executions of length $n$ in the concurrent system and paths of length $n$ in the \ADSC, provided the start and end nodes are of the given form.
\end{proof}

\begin{proposition}
  \label{prop:7}
  Let $(\M,X,\bot)$ be a concurrent system of characteristic root~$r$, and let $\rho$ be the spectral radius of\/~$\ADSC$. Then  $r=1/\rho$ if $\rho\neq0$ and $r=+\infty$ if $\rho=0$.
\end{proposition}

\begin{proof}
Let $F$ be the adjacency matrix of $\ADSC$. By definition, $\rho=\rho(F)$, see Section~\ref{sec:preliminaries}. For any two nodes $(\alpha,c)$ and $(\beta,d)$ of~$\DSC$, and for any two integers $i\in\{1,\dots,|c|\}$ and $j\in\{1,\dots,|d|\}$, the form of $F$ shows that, for all integers $n\geq0$:
\begin{gather}
  \label{eq:22}
  F_{(\alpha,c,i),(\beta,d,j)}^n=F_{(\alpha,c,1),(\beta,d,|d|)}^{n+(i-1)+(|d|-j)}\,;
\end{gather}
this is obvious when thinking of the number of paths represented by the two members of the above equality.

Therefore, if $R_{(\alpha,c,i),(\beta,d,j)}$ denotes the radius of convergence of
$$
\sum_{n\geq0}F^n_{(\alpha,c,i),(\beta,d,j)}z^n\,,
$$
and if $R=\min\bigl\{R_{(\alpha,c,i),(\beta,d,j)}\tqs\bigl((\alpha,c,i),(\beta,d,j)\bigr)\in\ADSC\times\ADSC\bigr\}$, one has:
\begin{gather}
  \label{eq:21}
R=\min\bigl\{R_{(\alpha,c,1),(\beta,d,|d|)}\tqs\bigl((\alpha,c),(\beta,d)\bigr)\in\DSC\times\DSC\bigr\}.
\end{gather}

By Proposition~\ref{prop:6}, we know that $\rho=1/R$, and it remains thus only to show that $r=R$. For any two node $(\alpha,c)$ and $(\beta,d)$ of $\DSC$, and for any integer $n\geq0$, it follows from Proposition~\ref{prop:5} that the number of paths of length $n$ in $\ADSC$ from $(\alpha,c,1)$ to $(\beta,d,|d|)$ is at most equal to $\card\M_{\alpha,\beta}(n)$. Therefore $R_{(\alpha,c,1),(\beta,d,|d|)}\geq r_{\alpha,\beta}\geq r$, which implies $R\geq r$ in view of~(\ref{eq:21}).

For the converse inequality, let $\alpha$ be any state and let $n$ be an integer. Then, still by Proposition~\ref{prop:5}, one has:
\begin{gather*}
  \card\M_{\alpha,\alpha}(n)=
\sum_{(c,d)\in\Cstar_\alpha\times\Cstar_\alpha}(F^n)_{(\alpha,c,1),(\alpha,d,|d|)}.
\end{gather*}
Since the series $\sum\card\M_{\alpha,\alpha}(n)r^n$ is divergent, it implies that at least one of the series $\sum (F^n)_{(\alpha,c,1),(\alpha,d,|d|)}r^n$ is divergent. Hence $R\leq r$, which completes the proof.
\end{proof}

\subsubsection{Positive and null nodes of\/ $\DSC$ and of\/ $\ADSC$}
\label{sec:positive-nul-nodes}

In the following definition, we denote by $C_1(x)$ the first clique that appears in the normal form of a non empty trace~$x$.

\begin{definition}
\label{def:10}
  Let $(\M,X,\bot)$ be a concurrent system, and let\/ $(\alpha,c)$ be a node of\/ \DSC. An execution $x\in\M_{\alpha}$ is an \emph{$(\alpha,c)$-protection} if:
  \begin{gather}
    \label{eq:18}
    \forall y\in\M_{\alpha\cdot x}\quad C_1(xy)=c.
  \end{gather}

  We say that $(\alpha,c)$ is a \emph{positive node} if there exists an $(\alpha,c)$-protection; we say that $(\alpha,c)$ is a \emph{null node} otherwise. We denote by $\DSC^+$ the sub-digraph of\/ $\DSC$ with all its positive nodes.

  A node $(\alpha,c,i)$ of\/ $\ADSC$ is \emph{positive} or \emph{null} according to whether $(\alpha,c)$ is a positive or a null node of\/ $\DSC$. We denote by\/ $\ADSC^+$ the sub-digraph of\/ $\ADSC$ with all its positive nodes.
\end{definition}

A necessary condition for $x$ to be an $(\alpha,c)$-protection is that $C_1(x)=c$; this is an application of~(\ref{eq:18}) with $y=\vd$.

The reader will find examples of concurrent systems with the description of the null and of the positive nodes in Section~\ref{sec:examples}.

Some elementary properties of positive and null nodes are gathered in the following proposition. A probabilistic interpretation of positive and null nodes will be given in Section~\ref{sec:null-nodes-dsc}.

\begin{proposition}
  \label{lem:2}
  Let $(\M,X,\bot)$ be a concurrent system.
  \begin{enumerate}
  \item\label{item:15} Let $\alpha$ be a state such that $\M_\alpha\neq\{\vd\}$. Then all pairs $(\alpha,c)$, with $c$ a maximal clique of~$\Cstar_\alpha$, are positive nodes of\/~$\DSC$. 
  \item\label{item:7} If $\bigl((\alpha,c),(\alpha',c')\bigr)$ is an arc in\/ $\DSC$ and if\/ $(\alpha',c')$ is a positive node of\/ $\DSC$, then $(\alpha,c)$ is also a positive node of\/~$\DSC$.
  \item\label{item:9} Let $\alpha$ be a state and let $x\in\M_\alpha$ be an execution starting from~$\alpha$. Then there exists an execution $y\in\M_{\alpha\cdot x}$ such that:
  \begin{inparaenum}[a)]
    \item $x\cdot y$ has same height as~$x$; and
    \item the nodes of the path in $\DSC$ corresponding to $x\cdot y$ all belong to $\DSC^+$.
    \end{inparaenum}
  \item\label{item:10} Let\/ $(\alpha,c)$ be a null node of\/ $\DSC$. Then for every $x\in\M_\alpha$ such that $C_1(x)=c$, there exists a letter $a\in\Sigma$ such that $x\in\M^a_\alpha$, where $\M^a=\langle\Sigma\setminus\{a\}\rangle$.
  \end{enumerate}
\end{proposition}

\begin{proof}
  Point~\ref{item:15}.\quad Any clique $c$ of $\Cstar_\alpha$ which is maximal in $\Cstar_\alpha$ with respect to the inclusion is itself an $(\alpha,c)$-protection. Hence $(\alpha,c)$ is a positive node.

  Point~\ref{item:7}.\quad Let $x'\in\M_{\alpha'}$ be an $(\alpha',c')$-protection. Put $x=c\cdot x'$. Then $x$ is an $(\alpha,c)$-protection.
  
  Point~\ref{item:9}.\quad  Let $h=\height(x)$ be the height of~$x$. Then choose $y$ as any maximal element in the finite and non empty set $P=\{z\in\M_{\alpha\cdot x}\tqs\height(x\cdot z)=h\}$. In particular the last clique of the normal form of $x\cdot y$ is maximal, hence the corresponding node of $\DSC$ is positive according to the remark stated above the proposition. If follows from point~\ref{item:7} already proved that all the previous nodes are all positive.

  Point~\ref{item:10}.\quad By contradiction, if there exists an execution $x\in\M_\alpha$ such that $C_1(x)=c$ and $x$ contains an occurrence of every letter, then for every execution $y\in\M_{\alpha\cdot x}$, the first clique of $x\cdot y$ would still be equal to~$c$. Hence $x$ would be an $(\alpha,c)$-protection, contradicting that $(\alpha,c)$ is a null node.
\end{proof}

For a trace monoid~$\M$, seen as a concurrent system \emph{via} the identification with the concurrent system~$\Mbar$ (see Definition~\ref{def:4}), positive and null nodes of $\DSC$ are easily determined, as shown by the following result (proof omitted).

\begin{proposition}
  \label{prop:8}
  Let $\M=\M(\Sigma,I)$ be a trace monoid, and let $\Mbar$ be the associated concurrent system. Let $\Sigma=\Sigma_1+\dots+\Sigma_p$ be the partition of\/ $\Sigma$ in connected components with respect to the relation $D=(\Sigma\times\Sigma)\setminus I$. Let $\M_i=\langle\Sigma_i\rangle$ for $i=1,\dots,p$, with set of cliques~$\C_i$.

  Then $\M$ identifies with the direct product $\M_1\times\dots\times\M_p$ and $\C$ identifies with the Cartesian product $\C_1\times\dots\times\C_p$.

  The positive nodes of\/ $\DSC$ are those non empty cliques $c$ of\/ $\C$ which correspond to a tuple $(c_1,\dots,c_p)\in\C_1\times\dots\times\C_p$ such that $c_i\neq\vd$ for all $i=1,\dots,p$. In particular, if $\M$ is irreducible, all the nodes of\/ $\DSC$ are positive.
\end{proposition}

\subsection{Examples}
\label{sec:examples}

We collect a few examples to illustrate the notions introduced above. The positive and null nodes can be determined ``by hand'' using Definition~\ref{def:10} and Proposition~\ref{lem:2}. Another and simpler method to determine the positive and null nodes will be to use the fothcoming results of Section~\ref{sec:appl-spectr-prop}. 

\subsubsection{Example of trace monoids}
\label{sec:example-trace-mono}


The smallest non-elementary irreducible monoid is $\M=\langle a,b,c\;|\; ab=ba\rangle$. The set of cliques is $\C=\{\vd,a,b,c,ab\}$, and the M\"obius polynomial is $\mu(z)=1-3z+z^2$, with characteristic root $r=(3-\sqrt5)/2$. 

Figure~\ref{fig:01},~$(a)$ depicts the Coxeter graph of the monoid (see Section~\ref{sec:trace-monoids}). 
Let $\Mbar=(\M,\{*\},\bot)$ be the concurrent system associated with~$\M$ as in Definition~\ref{def:4}. Figure~\ref{fig:01},~$(b)$ depicts the \DSC\ of~$\Mbar$, which is isomorphic to the digraph of cliques of~$\M$.

\begin{figure}
  \centering
  \begin{tabular}{c|c} \begin{minipage}[c]{.4\textwidth}
$$      \xymatrix{\bullet\POS!L(2)\drop{a}\POS[]\ar@{-}[r]&\bullet\POS!U(1.5)\drop{c}\POS[]\ar@{-}[r]&\bullet\POS!R(1.5)\drop{b}}
$$
 \end{minipage}
    &
\begin{minipage}[c]{.6\textwidth}
$$      \xymatrix@R=1.5em{
    &*+[F]{\strut(*,ab)}\POS!U!L(0.5)\ar@(ul,ur)!U!R(0.5)[]%
      \POS[]\POS!L\ar[dl]!U\POS[]\POS!R\ar[dr]!U
\\
*+[F]{\strut(*,a)}\POS!L!U(0.5)\ar@(ul,dl)[]!L!D(0.5)
\POS[]\POS!D\ar@{<->}[dr]!L
&&*+[F]{\strut(*,b)}\POS!R!U(0.5)\ar@(ur,dr)[]!R!D(0.5)
\POS[]\POS!D\ar@{<->}[dl]!R
\\
&    *+[F]{\strut(*,c)}%
\POS!D!L(0.5)\ar@(dl,dr)!D!R(0.5)[]
\POS[]\ar@{<->}[uu]
}
$$
\end{minipage}
    \\
    \\[1em]
$(a)$&$(b)$
  \end{tabular}
  \caption{\small$(a)$---Coxeter graph of the trace monoid $\langle a,b,c\;|\; ab=ba\rangle$.\quad$(b)$---Digraph of cliques of the same monoid (some arrows have a double tip).}
  \label{fig:01}
\end{figure}
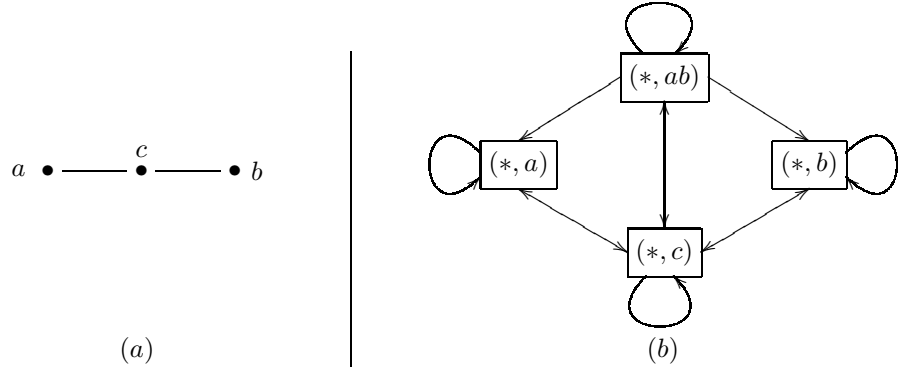

\subsubsection{Elementary example of an irreducible concurrent system}
\label{sec:an-elem-example}

The example depicted on Figure~\ref{fig:elemeytartas} can be easily analyzed. The digraph of states is depicted on Figure~\ref{fig:elemeytartas},~$(a)$. The trace monoid is represented by its Coxeter graph on Figure~\ref{fig:elemeytartas},~$(b)$. The associated concurrent system is irreducible. Its M\"obius matrix is:
\begin{gather*}
  \mu(z)=
  \begin{array}{c}
    \alpha_0\\\alpha_1
  \end{array}
  \begin{pmatrix}
    1-2z+z^2&-z+z^2\\
    -z&1-z
  \end{pmatrix}
\end{gather*}
with determinant $\theta(z)=(1-z)(1-2z)$. The characteristic root is thus $r=1/2$. The $\DSC$ is depicted in Figure~\ref{fig:elemeytartas},~$(c)$. It has one null node only. The positive nodes form a unique strongly connected component.

\begin{figure}
  \centering
  \begin{tabular}{c|c|c}
           \begin{minipage}[c]{.2\textwidth}
    $$
\xymatrix{*++[o][F-]{\alpha_0}%
\POS!L\ar@(l,ul)^{d}[]!U!L(.2)%
\POS!R\ar@(r,ur)_{a}[]!U!R(.2)%
\POS[]\ar@<1ex>^{b}[d]\ar@<1ex>^{c}[d];[]\\
  *++[o][F-]{\alpha_1}
\POS!L\ar@(l,dl)_{d}[]!D
    }
    $$
  \end{minipage}
    &
\begin{minipage}[c]{.1\textwidth}
         $$
         \xymatrix@C=1em{
           \bullet\ar@{-}[r]\ar@{-}[d]\labelu{a}&\bullet\ar@{-}[dl]\labelu{b}\\
           \bullet\ar@{-}[d]\labelr{c}\\
           \bullet\labelr{d}
}
$$
\end{minipage}
    &
         \begin{minipage}[c]{.7\textwidth}
   $$
    \xymatrix@C=1em{
     & &&*+[F]{(\alpha_0,d)}\POS!R!D(.5)\ar@(dr,ur)[]!R!U(0.5)&&{}\save[]*\txt<2em>{\itshape null node}\restore\\
 {}\POS[]+<-2.6em,1.5em>\ar@{--}[rrrrr]+<1.6em,1.5em>      &*+[F]{(\alpha_0,ad)}\POS!U\ar[urr]!L\POS[]\ar[rrr]\POS!L\ar[dl]!U\POS[]\ar[d]\POS[]\POS!L!D(0.5)\ar@(dl,dr)[]!D!L(0.5)
      &&&*+[F]{(\alpha_0,bd)}\POS!D\ar[dl]!R\ar[ddl]!R&\\
      *+[F]{(\alpha_0,a)}\ar[r]\POS!L!D(0.5)\ar@(dl,dr)[]!D!L(0.5)
      &*+[F]{(\alpha_0,b)}     &&*+[F]{(\alpha_1,c)}\POS!U!R(0.25)\ar[ur]!D!L\POS[]\ar@<1ex>[ll]\ar@<1ex>[ll];[]\POS!L!U\ar[ull]!R!D
      \POS[]\ar'[u][uu]\POS!D!L(0.5)\ar@(dl,dr)[lll]!D!R(0.5)
      &{}\save[]\POS[]+<2em,1em>*\txt<2em>{\itshape positive nodes}\restore
 \\
 &&&*+[F]{(\alpha_1,d)}\POS!L!D(0.5)\ar@(dl,dr)[]!D!L(0.5)\POS[]\ar[u]&\strut\hspace{3em}\strut
       \save "1,4"!L(1.5)."1,4"!R(1.5)*+[F.]\frm{}\restore
       \save "2,1"+<-2em,1.2em>."4,5"!R(1.15)*+[F.]\frm{}\restore
      }
     $$
   \end{minipage}
\\[9em]
    $(a)$&$(b)$&$(c)$
  \end{tabular}                                                                              
  \caption{\small $(a)$---The labelled digraph of states of an irreducible concurrent system with two states.\quad$(b)$---The Coxeter graph of the associated trace monoid.\quad $(c)$---$\DSC$ of the associated concurrent system. Light dash boxes depict the strongly connected components. The long dash line marks the frontier between positive and null nodes.}
\label{fig:elemeytartas}
\end{figure}
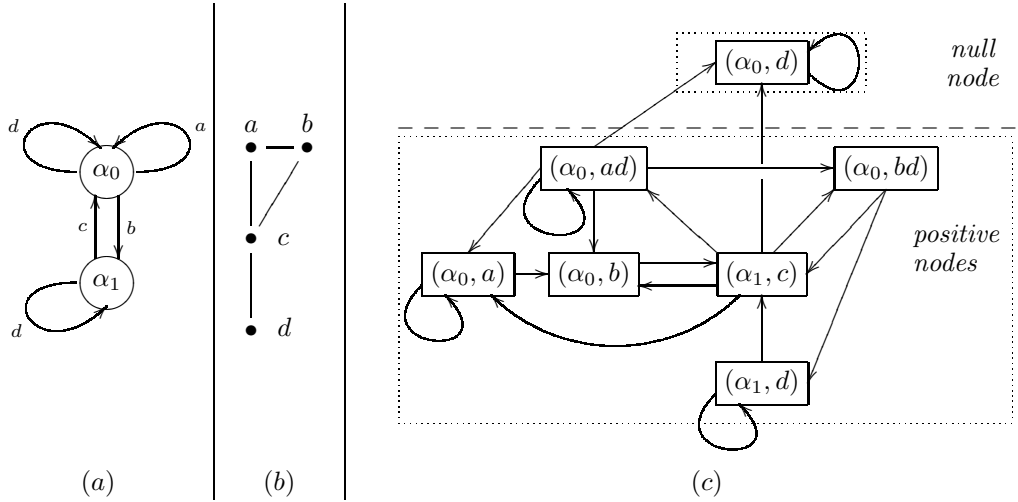

\subsubsection{Tilings of the Aztec diamond: an intersting example}
\label{sec:tilings}

\begin{figure}
  \centering
\begin{tabular}{cccc}
  \xymatrix@R=1.5em@C=1.5em@M=0pt{
{}    &{}&{}&{}&{}\\
{}    &{}&{}\drop{\bullet}\POS!L(2)\drop{a}&{}&{}\\
{}    &{}&{}&{}&{}\\
{}    &{}&{}\drop{\bullet}\POS!L(2)\drop{b}&{}&{}\\
{}    &{}&{}&{}&{}
  \ar@{-}"1,2";"3,2"
  \ar@{-}"3,2";"3,3"
  \ar@{-}"3,3";"1,3"
  \ar@{-}"1,3";"1,2"
  \ar@{-}"1,3";"1,4"
  \ar@{-}"1,4";"3,4"
  \ar@{-}"3,4";"3,3"
  \ar@{-}"3,3";"1,3"
  \ar@{-}"3,2";"3,3"
  \ar@{-}"3,3";"5,3"
  \ar@{-}"5,3";"5,2"
  \ar@{-}"5,2";"3,2"
  \ar@{-}"3,3";"3,4"
  \ar@{-}"3,4";"5,4"
  \ar@{-}"5,4";"5,3"
  \ar@{-}"5,3";"3,3"
  \ar@{-}"2,1";"2,2"
  \ar@{-}"2,2";"4,2"
  \ar@{-}"4,2";"4,1"
  \ar@{-}"4,1";"2,1"
  \ar@{-}"2,4";"2,5"
  \ar@{-}"2,5";"4,5"
  \ar@{-}"4,5";"4,4"
  \ar@{-}"4,4";"2,4"
}
&

  \xymatrix@R=1.5em@C=1.5em@M=0pt{
{}    &{}&{}&{}&{}\\
{}    &{}&{}\drop{\bullet}\POS!U(3)\drop{a}&{}&{}\\
{}    &{}&{}&{}&{}\\
{}    &{}&{}\drop{\bullet}\POS!L(2)\drop{b}&{}&{}\\
{}    &{}&{}&{}&{}
  \ar@{-}"1,2";"1,4"
  \ar@{-}"1,2";"2,2"
  \ar@{-}"2,2";"2,4"
  \ar@{-}"2,4";"1,4"
  \ar@{-}"2,2";"2,4"
  \ar@{-}"2,4";"3,4"
  \ar@{-}"3,4";"3,2"
  \ar@{-}"3,2";"2,2"
  \ar@{-}"3,2";"3,3"
  \ar@{-}"3,3";"5,3"
  \ar@{-}"5,3";"5,2"
  \ar@{-}"5,2";"3,2"
  \ar@{-}"3,3";"3,4"
  \ar@{-}"3,4";"5,4"
  \ar@{-}"5,4";"5,3"
  \ar@{-}"5,3";"3,3"
  \ar@{-}"2,1";"2,2"
  \ar@{-}"2,2";"4,2"
  \ar@{-}"4,2";"4,1"
  \ar@{-}"4,1";"2,1"
  \ar@{-}"2,4";"2,5"
  \ar@{-}"2,5";"4,5"
  \ar@{-}"4,5";"4,4"
  \ar@{-}"4,4";"2,4"
}
&
  \xymatrix@R=1.5em@C=1.5em@M=0pt{
{}    &{}&{}&{}&{}\\
{}    &{}&{}\drop{\bullet}\POS!L(2)\drop{a}&{}&{}\\
{}    &{}&{}&{}&{}\\
{}    &{}&{}\drop{\bullet}\POS!U(3)\drop{b}&{}&{}\\
{}    &{}&{}&{}&{}
  \ar@{-}"1,2";"3,2"
  \ar@{-}"3,2";"3,3"
  \ar@{-}"3,3";"1,3"
  \ar@{-}"1,3";"1,2"
  \ar@{-}"1,3";"1,4"
  \ar@{-}"1,4";"3,4"
  \ar@{-}"3,4";"3,3"
  \ar@{-}"3,3";"1,3"
  \ar@{-}"4,2";"4,4"
  \ar@{-}"4,4";"5,4"
  \ar@{-}"5,4";"5,2"
  \ar@{-}"5,2";"4,2"
  \ar@{-}"2,1";"2,2"
  \ar@{-}"2,2";"4,2"
  \ar@{-}"4,2";"4,1"
  \ar@{-}"4,1";"2,1"
  \ar@{-}"2,1";"2,2"
  \ar@{-}"2,2";"4,2"
  \ar@{-}"4,2";"4,1"
  \ar@{-}"4,1";"2,1"
  \ar@{-}"2,4";"2,5"
  \ar@{-}"2,5";"4,5"
  \ar@{-}"4,5";"4,4"
  \ar@{-}"4,4";"2,4"
}

&
  \xymatrix@R=1.5em@C=1.5em@M=0pt{
{}    &{}&{}&{}&{}\\
{}    &{}&{}\drop{\bullet}\POS!U(3)\drop{a}&{}&{}\\
{}    &{}&{}\drop{\bullet}\POS!U(3)\drop{c}&{}&{}\\
{}    &{}&{}\drop{\bullet}\POS!U(3)\drop{b}&{}&{}\\
{}    &{}&{}&{}&{}
  \ar@{-}"1,2";"1,4"
  \ar@{-}"1,2";"2,2"
  \ar@{-}"2,2";"2,4"
  \ar@{-}"2,4";"1,4"
  \ar@{-}"2,2";"2,4"
  \ar@{-}"2,4";"3,4"
  \ar@{-}"3,4";"3,2"
  \ar@{-}"3,2";"2,2"
  \ar@{-}"3,2";"3,4"
  \ar@{-}"3,4";"4,4"
  \ar@{-}"4,4";"4,2"
  \ar@{-}"4,2";"3,2"
  \ar@{-}"4,2";"4,4"
  \ar@{-}"4,4";"5,4"
  \ar@{-}"5,4";"5,2"
  \ar@{-}"5,2";"4,2"
  \ar@{-}"2,1";"2,2"
  \ar@{-}"2,2";"4,2"
  \ar@{-}"4,2";"4,1"
  \ar@{-}"4,1";"2,1"
  \ar@{-}"2,4";"2,5"
  \ar@{-}"2,5";"4,5"
  \ar@{-}"4,5";"4,4"
  \ar@{-}"4,4";"2,4"
}
  \\
  $0$&$1$&$2$&$3$\\
    \xymatrix@R=1.5em@C=1.5em@M=0pt{
{}    &{}&{}&{}&{}\\
{}    &{}&{}&{}&{}\\
{}    &{}\drop{\bullet}\POS!L(2)\drop{d}&{}\drop{\bullet}\POS!L(2)\drop{c}&{}\drop{\bullet}\POS!L(2)\drop{e}&{}\\
{}    &{}&{}&{}&{}\\
{}    &{}&{}&{}&{}
  \ar@{-}"2,2";"4,2"
  \ar@{-}"4,2";"4,3"
  \ar@{-}"4,3";"2,3"
  \ar@{-}"2,3";"2,2"
  \ar@{-}"2,3";"4,3"
  \ar@{-}"4,3";"4,4"
  \ar@{-}"4,4";"2,4"
  \ar@{-}"2,4";"2,3"
  \ar@{-}"2,1";"2,2"
  \ar@{-}"2,2";"4,2"
  \ar@{-}"4,2";"4,1"
  \ar@{-}"4,1";"2,1"
  \ar@{-}"2,4";"2,5"
  \ar@{-}"2,5";"4,5"
  \ar@{-}"4,5";"4,4"
  \ar@{-}"4,4";"2,4"
  \ar@{-}"1,2";"1,4"
  \ar@{-}"1,2";"2,2"
  \ar@{-}"2,2";"2,4"
  \ar@{-}"2,4";"1,4"
  \ar@{-}"4,2";"4,4"
  \ar@{-}"4,4";"5,4"
  \ar@{-}"5,4";"5,2"
  \ar@{-}"5,2";"4,2"
}
      &
            \xymatrix@R=1.5em@C=1.5em@M=0pt{
{}    &{}&{}&{}&{}\\
{}    &{}&{}&{}&{}\\
{}    &{}\drop{\bullet}\POS!L(2)\drop{d}&{}&{}\drop{\bullet}\POS!U(2)\drop{e}&{}\\
{}    &{}&{}&{}&{}\\
{}    &{}&{}&{}&{}
  \ar@{-}"2,3";"2,5"
  \ar@{-}"2,5";"3,5"
  \ar@{-}"3,5";"3,3"
  \ar@{-}"3,3";"2,3"
  \ar@{-}"3,3";"3,5"
  \ar@{-}"3,5";"4,5"
  \ar@{-}"4,5";"4,3"
  \ar@{-}"4,3";"3,3"
  \ar@{-}"2,2";"4,2"
  \ar@{-}"4,2";"4,3"
  \ar@{-}"4,3";"2,3"
  \ar@{-}"2,3";"2,2"
  \ar@{-}"2,1";"2,2"
  \ar@{-}"2,2";"4,2"
  \ar@{-}"4,2";"4,1"
  \ar@{-}"4,1";"2,1"
  \ar@{-}"1,2";"1,4"
  \ar@{-}"1,2";"2,2"
  \ar@{-}"2,2";"2,4"
  \ar@{-}"2,4";"1,4"
  \ar@{-}"4,2";"4,4"
  \ar@{-}"4,4";"5,4"
  \ar@{-}"5,4";"5,2"
  \ar@{-}"5,2";"4,2"
}
      &
            \xymatrix@R=1.5em@C=1.5em@M=0pt{
{}    &{}&{}&{}&{}\\
{}    &{}&{}&{}&{}\\
{}    &{}\drop{\bullet}\POS!U(3)\drop{d}&{}&{}\drop{\bullet}\POS!L(2)\drop{e}&{}\\
{}    &{}&{}&{}&{}\\
{}    &{}&{}&{}&{}
  \ar@{-}"2,1";"2,3"
  \ar@{-}"2,3";"3,3"
  \ar@{-}"3,3";"3,1"
  \ar@{-}"3,1";"2,1"
  \ar@{-}"3,1";"3,3"
  \ar@{-}"3,3";"4,3"
  \ar@{-}"4,3";"4,1"
  \ar@{-}"4,1";"3,1"
  \ar@{-}"2,3";"2,4"
  \ar@{-}"2,4";"4,4"
  \ar@{-}"4,4";"4,3"
  \ar@{-}"4,3";"2,3"
  \ar@{-}"2,4";"2,5"
  \ar@{-}"2,5";"4,5"
  \ar@{-}"4,5";"4,4"
  \ar@{-}"4,4";"2,4"
  \ar@{-}"1,2";"1,4"
  \ar@{-}"1,2";"2,2"
  \ar@{-}"2,2";"2,4"
  \ar@{-}"2,4";"1,4"
  \ar@{-}"4,2";"4,4"
  \ar@{-}"4,4";"5,4"
  \ar@{-}"5,4";"5,2"
  \ar@{-}"5,2";"4,2"
}
      &
              
            \xymatrix@R=1.5em@C=1.5em@M=0pt{
{}    &{}&{}&{}&{}\\
{}    &{}&{}&{}&{}\\
{}    &{}\drop{\bullet}\POS!U(3)\drop{d}&{}&{}\drop{\bullet}\POS!U(3)\drop{e}&{}\\
{}    &{}&{}&{}&{}\\
{}    &{}&{}&{}&{}
  \ar@{-}"2,1";"2,3"
  \ar@{-}"2,3";"3,3"
  \ar@{-}"3,3";"3,1"
  \ar@{-}"3,1";"2,1"
  \ar@{-}"3,1";"3,3"
  \ar@{-}"3,3";"4,3"
  \ar@{-}"4,3";"4,1"
  \ar@{-}"4,1";"3,1"
  \ar@{-}"3,3";"3,5"
  \ar@{-}"3,5";"4,5"
  \ar@{-}"4,5";"4,3"
  \ar@{-}"4,3";"3,3"
  \ar@{-}"2,3";"2,5"
  \ar@{-}"2,5";"3,5"
  \ar@{-}"3,5";"3,3"
  \ar@{-}"3,3";"2,3"
  \ar@{-}"1,2";"1,4"
  \ar@{-}"1,2";"2,2"
  \ar@{-}"2,2";"2,4"
  \ar@{-}"2,4";"1,4"
  \ar@{-}"4,2";"4,4"
  \ar@{-}"4,4";"5,4"
  \ar@{-}"5,4";"5,2"
  \ar@{-}"5,2";"4,2"
}
  \\
  $3'$&$2'$&$1'$&$0'$
\end{tabular}

\caption{\small The $8$ tilings of the Aztec diamond of order $2$}
  \label{fig:examples03}
\end{figure}
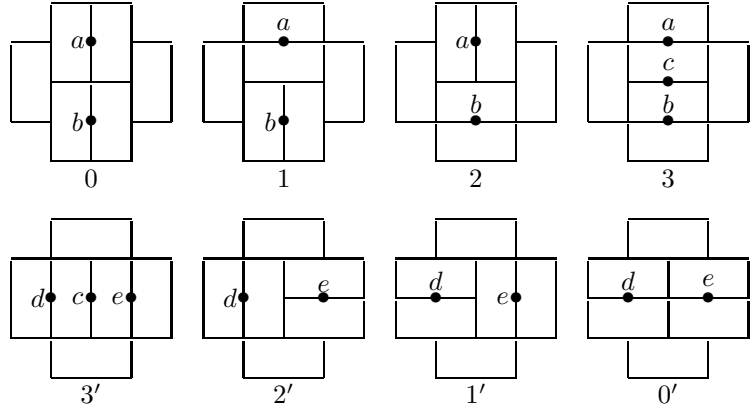

\begin{figure}
  \centering
  \begin{tabular}{c|c}
   \begin{minipage}[c]{.5\textwidth}
$$   \xymatrix@R=1.5em@C=1.2em@M=.2em{
    &\bullet\ar@{-}_{b}[dl]\ar@{-}^{a}[dr]\POS!U(2)\drop{2}&&&&\bullet\ar@{-}^{d}[dr]\POS!U(2)\drop{2'}\\
    \bullet\POS!L(2)\drop{0}&&\bullet\ar@{-}^{c}[rr]\POS!R!U(2)\drop{3}&&
\bullet\ar@{-}_{d}[dr]\ar@{-}^{e}[ur]\POS!U(2)!L\drop{3'} &&\bullet\POS!R(2)\drop{0'}\\
    &\bullet\ar@{-}^{a}[ul]\ar@{-}_{b}[ur]
\POS!D(2)\drop{1}&&&&\bullet\ar@{-}_{e}[ur]\POS!D(2)\drop{1'}
}$$
\end{minipage}
                             &
   \begin{minipage}[c]{.5\textwidth}
$$    \xy{
{\xypolygon5"A"{~:{(2,2):(3,3):}~><{@{}}~={0}{\hbox{\strut$\;\bullet\;$}}%
}},%
{\xypolygon5"B"{~:{(2,2):(4,4):}~><{@{}}~={0}{}}},
"B1"*{c},"B2"*{a},"B3"*{b},"B4"*{d},"B5"*{e},
"A1";"A2"**@{-},
"A1";"A3"**@{-},
"A1";"A4"**@{-},
"A1";"A5"**@{-},
"A3";"A4"**@{-},
"A2";"A4"**@{-},
"A2";"A5"**@{-},
"A3";"A5"**@{-},
}%
                               \endxy
                               $$
                             \end{minipage}
    \\
    $(a)$&$(b)$
  \end{tabular}
  \caption{\small $(a)$---Graph of tilings of the Aztec diamond of order~$2$. Refer to  Figure~\ref{fig:examples03} for the numbering of tilings. \quad$(b)$---Coxeter graph of the monoid acting on the tilings.\quad }
  \label{fig:pqijwqwd}
\end{figure}
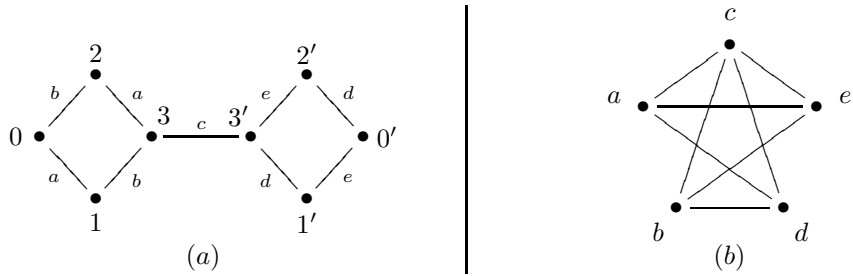

\begin{figure}
  \input{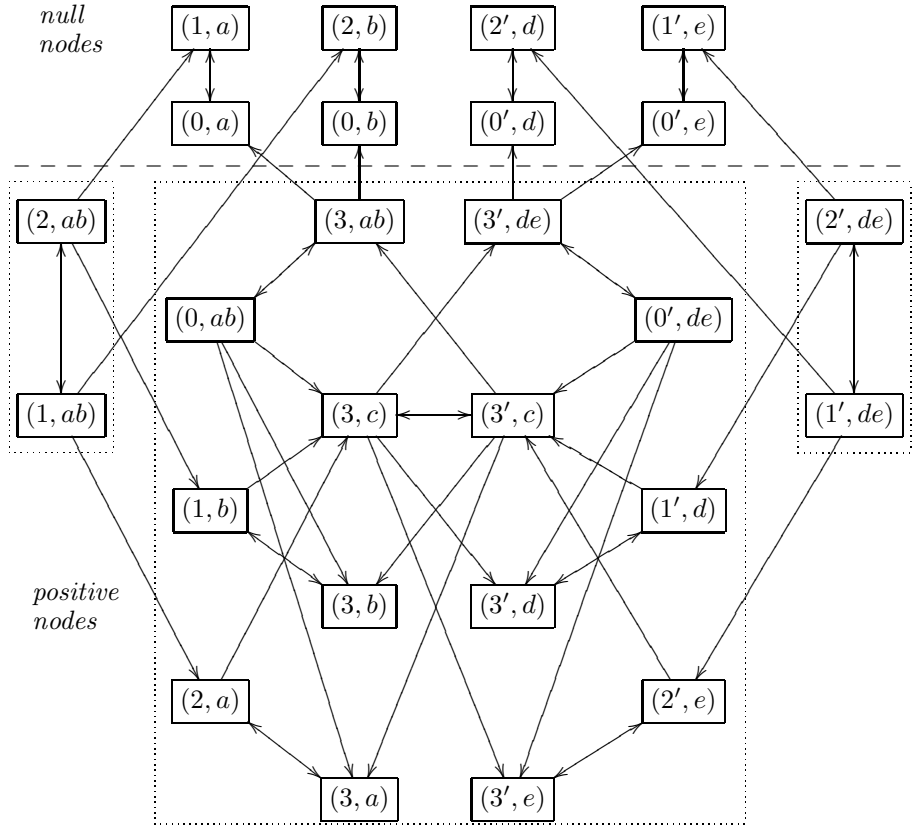}
  \caption{\small $\DSC$ of the concurrent system associated with the tilings of the Aztec diamond of order~$2$. 
 The long dash line marks the frontier between positive and null nodes. Dashed boxes show the strongly connected components of the~$\DSC^+$.}
  \label{fig:dsctiling}
\end{figure}

 Consider the set $X$ of tilings of the Aztec diamond of order $2$ by $1\times 2$ and $2\times 1$ dominoes, depicted in Figure~\ref{fig:examples03}. For each tiling, we have put a labeled bullet to denote the pair of dominoes which can be rotated. This defines an action of the free monoid generated by $a$, $b$, $c$, $d$, $e$ on $X\cup\{\bot\}$. For instance: $0\cdot a=1$, $0\cdot b=2$ and $0\cdot c=\bot$. It is a simple verification that this action factorizes through the commutativity relations associated to rotations acting on disjoint dominoes, \ie,~$(a,b)$ and~$(d,e)$; whence a concurrent system $(\M,X,\bot)$ with $\M=\langle a,b,c,d,e\;|\; ab=ba,\ de=ed\rangle$. The multigraph of states, where all arcs are doubly-oriented and represented as undirected edges, is depicted in Figure~\ref{fig:pqijwqwd},~$(a)$, and the Coxeter graph of the monoid is depicted on Figure~\ref{fig:pqijwqwd},~$(b)$. The concurrent system $(\M,X,\bot)$ is irreducible, see  Figure~\ref{fig:pqijwqwd}. 

The $\DSC$ is depicted in Figure~\ref{fig:dsctiling}. It has 26 nodes among which the $8$ following null nodes:
$(0,a),\quad(0,b),\quad(1,a),\quad(2,b),\quad(0',d),\quad(0',e),\quad(1',e),\quad(2',d).$

The $\DSC^+$ has $3$ strongly connected components and only 1 terminal strongly connected component. The characteristic root of the system will be computed in Section~\ref{sec:tiling-example}.

\subsubsection{An example with several terminal components of\/ $\DSC^+$}
\label{sec:sever-term-comp}

All previous examples have the common feature that their $\DSC^+$ has a unique terminal component. Maybe contradicting our intuition, this is actually not always the case, as shown by the following example. 

Let $(\M,X,\bot)$ be the concurrent system defined by $X=\{0,1,\ldots,11\}$, $\M=\langle a,b,c,d,e,f\;|\;ab=ba,\ ad=da,\ bf=fb,\ cd=dc,\ ce=ec,\ ef=fe\rangle$, and with the graph of states depicted in Figure~\ref{fig:pojqwpojq},~top. 

An important point is to verify that this graph defines indeed a concurrent system for the trace monoid~$\M$, \ie, that the action of $\Sigma^*$ on $X\cup\{\bot\}$ has a well defined quotient with respect to the six commutation relations; this is indeed the case. The concurrent system is clearly accesible and alive; the trace monoid is irreducible since its Coxeter graph, depicted in Figure~\ref{fig:pojqwpojq}, bottom~$(a)$, is connected. Hence the concurrent system is irreducible; yet the $\DSC^+$, shown in Figure~\ref{fig:pojqwpojq}, bottom~$(b)$, has 2 terminal strongly connected components.

\begin{figure}
  \begin{tabular}{c|c}
    \multicolumn{2}{c}{%
    \begin{minipage}{\textwidth}
      $$
  \xymatrix@M=.2ex{
    \bullet\POS!U(2.6)*++[F]{1}&\bullet\ar[l]_{b}\POS!U(2.6)*++[F]{0}&\bullet\ar[l]_{e}\labelu{11}\\
    \bullet\ar[u]^{f}\POS!U(2.2)!R(2)\drop{10}
    &\bullet\ar[l]^{b}\ar[u]^{f}\POS!U(2.2)!R(2)\drop{9}
    &\bullet\ar[l]^{e}\ar[u]^{f}\POS!U(2.2)!R(2)\drop{8}
    &\bullet\ar[l]^{d}\POS!U(2.2)\drop{7}\\
    &\bullet\ar[u]^{c}\POS!U(2.2)!R(2)\drop{6}
    &\bullet\ar[l]^{e}\ar[u]^{c}\POS!U(2.2)!R(2)\drop{5}
    &\bullet\ar[l]^{d}\ar[u]^{c}\POS!U(2.2)!R(2)\drop{4}
    &\bullet\ar[l]^{b}\POS!U(2.2)\drop{3}\\
    &&\bullet\ar[u]^{a}\POS!D(2)\drop{2}
    &\bullet\ar[l]^{d}\ar[u]^{a}\POS!D(2.2)*++[F]{1}
    &\bullet\ar[l]^{b}\ar[u]^{a}\POS!D(2.2)*++[F]{0}
    }
    $$
  \end{minipage}
      }\\
    \multicolumn{2}{c}{}\\
    \hline\\
    \begin{minipage}[c]{.5\textwidth}
\hfill      \xy{
{\xypolygon6"A"{~:{(2,2):(3,3):}~><{@{}}~={0}{\hbox{\strut$\;\bullet\;$}}%
}},%
{\xypolygon6"B"{~:{(2,2):(4,4):}~><{@{}}~={0}{}}},
"B1"*{c},"B2"*{b},"B3"*{a},"B4"*{f},"B5"*{e},"B6"*{d},
"A1";"A2"**@{-},
    "A1";"A3"**@{-},
    "A1";"A4"**@{-},
    "A2";"A5"**@{-},
    "A2";"A6"**@{-},
    "A3";"A4"**@{-},
    "A3";"A5"**@{-},
    "A4";"A6"**@{-},
    "A5";"A6"**@{-},
}%
\endxy\hfill\strut
\end{minipage}
                     &
    \begin{minipage}[c]{.5\textwidth}
\noindent\strut\hfill
                   \xymatrix{
    *+[F]{(8,ef)}\POS!L\ar@(dl,ul)[dd]!L\\
    *+[F]{(4,cd)}\ar[u]\\
    *+[F]{(0,ab)}\ar[u]
      }\hfill
                           \xymatrix{
    *+[F]{(9,bf)}\POS!R\ar@(dr,ur)[dd]!R\\
    *+[F]{(5,ec)}\ar[u]\\
    *+[F]{(1,ad)}\ar[u]
  }%
  \hfill\strut
  \end{minipage}
    \\
  $(a)$&$(b)$
  \end{tabular}
  \caption{\small \emph{Top}---Graph of states of a concurrent system. The bullets labeled by $\{0,\ldots,11\}$ represent the 12 states of the system. States labeled with the identical framed labels \usebox{\labelz}  and \usebox{\labelun} are identified.\quad\emph{Bottom,~$(a)$}---Coxeter graph of the trace monoid.\quad
    \emph{Bottom,~$(b)$}---$\DSC^+$ of the concurrent system.}
  \label{fig:pojqwpojq}
\end{figure}
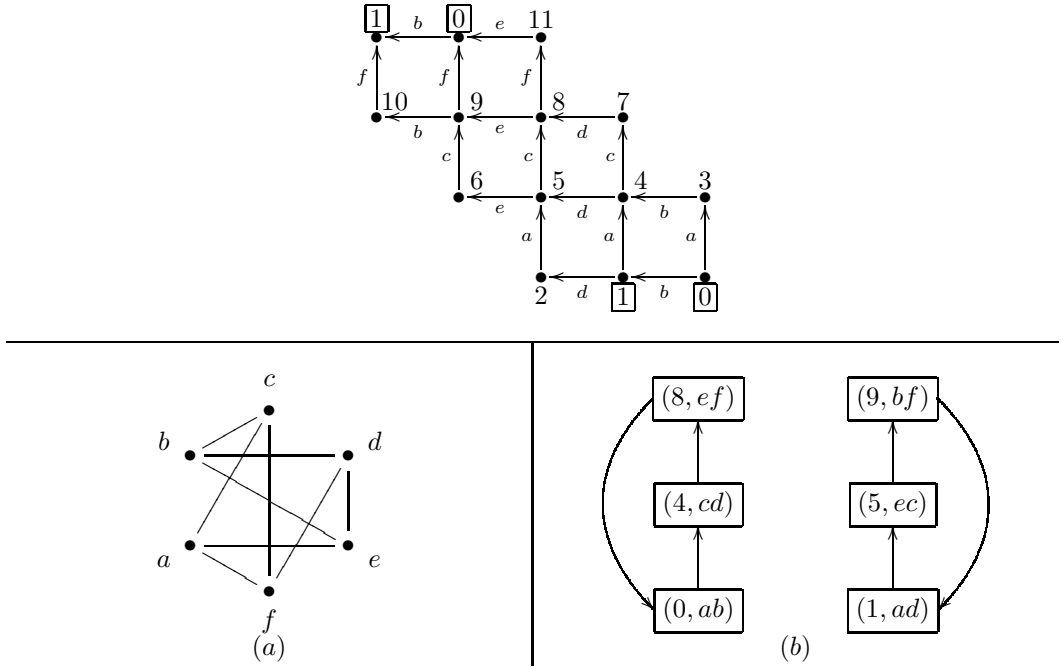

\section{The spectral property}
\label{sec:psectr-prop-conc}

\subsection{Definition of the spectral property. The case of trace monoids}
\label{sec:defin-spectr-prop}

Let $(\M,X,\bot)$ be a concurrent system. Let $a$ be a letter in the base alphabet $\Sigma$ of~$\M$. Put $\Sigma^a=\Sigma\setminus\{a\}$, and let $\M^a$ be the submonoid of $\M$ generated by~$\Sigma^a$. Then the action of $\M$ on $X\cup\{\bot\}$ induces by restriction an action of $\M^a$ on~$X\cup\{\bot\}$, such that $(\M^a,X,\bot)$ is still a concurrent system.

\begin{definition}
  \label{def:2}
  Let   $\SS=(\M,X,\bot)$ be a concurrent system of characteristic root~$r$ and with base alphabet\/~$\Sigma$. For every letter~$a$, let $r^a$ be the characteristic root of\/ $(\M^a,X,\bot)$. 
We say that $\SS$ has the \emph{spectral property} if $r^a>r$ for every $a\in\Sigma$.
\end{definition}

We append the exponent $a$ to the symbols introduced at the beginning of Section~\ref{sec:comb-conc-syst} and in Definition~\ref{def:3} to denote the objects corresponding to $(\M^a,X,\bot)$. The inclusions $\M^a_{\alpha,\beta}(n)\subseteq\M_{\alpha,\beta}(n)$, valid for all integers~$n$, imply that $r^a_{\alpha,\beta}\geq r$ holds for every letter $a$ and for all states $\alpha$ and~$\beta$. The spectral property states that these inequalities are all strict.

For a trace monoid, the equivalence between irreducibility of the trace monoid and the spectral property can be seen as a direct consequence of Perron-Frobenius theory. Yet, the proof that we give below is original, to the best of our knowledge, and does not appeal to Perron-Frobenius theory. In our opinion, it is interesting in that it does not require any knowledge on the structure of the digraph of cliques of the monoid---a feature that we wish to extend for general concurrent systems.

\begin{proposition}
  \label{prop:1}
Let $\M=\M(\Sigma,I)$ be a trace monoid with M\"obius polynomial $\mu(z)$ and characteristic root~$r$. Then:
\begin{enumerate}
\item\label{item:1} Let $\Sigmatilde$ be a subset of\/~$\Sigma$, and let $\rtilde$ be the characteristic root of $\Mtilde=\langle\Sigmatilde\rangle$. Then $r\leq \rtilde$. Furthermore, $r<\rtilde$ if and only if $\mu(\rtilde)>0$.
\item\label{item:2} $\M$ has the spectral property if and only if $\M$ is irreducible.
\end{enumerate}
\end{proposition}

\begin{proof}
  Point~\ref{item:1}.  Let $\Gtilde(z)$ be the growth series of~$\Mtilde$, and let $\mutilde(z)$ be the M\"obius polynomial of~$\Mtilde$. The inequality $r\leq\rtilde$ derives from the inclusions $\Mtilde(n)\subseteq\M(n)$, true for all integers~$n$, which imply that $\Gtilde(t)\leq G(t)<\infty$ for all $t\in(0,r)$. The identity $\Gtilde(t)\mutilde(t)=1$ holds in the field of reals for all $t\in[0,\rtilde)$. Since $r\leq \rtilde$, it follows firstly that $\mutilde(r)\geq0$, and secondly that $r=\rtilde$ if and only if $\mutilde(r)=0$, whence the stated equivalence.

  Point~\ref{item:2}. Assuming that $\M$ does not have the spectral property, we prove that $\M$ is not irreducible. Pick a letter $a$ such that $r=r^a$. Then $\mu^a(r)=0$, where $\mu^a(z)$ denotes the M\"obius polynomial of~$\M^a$. Let $D(a)=\{b\in\Sigma\tqs(a,b)\notin I\}$, and let $\M^{D(a)}=\langle\Sigma\setminus D(a)\rangle$, with M\"obius polynomial $\mu^{D(a)}(z)$. The definition of the M\"obius polynomial (see Section~\ref{sec:trace-monoids}) shows that:
  \begin{gather}
    \label{eq:1}
    \mu(z)=\mu^a(z)-z\mu^{D(a)}(z).
  \end{gather}
From~(\ref{eq:1}), and from $\mu(r)=\mu^a(r)=0$, follows that $\mu^{D(a)}(r)=0$. Hence Point~\ref{item:1} already proved shows that the characteristic root of $\M^{D(a)}$ is~$r$. Let $\Sigmatilde$ be the subset of $\Sigma$ containing all letters but those in the connected component of $a$ in the graph $(\Sigma,D)$, and let $\Mtilde=\langle\Sigmatilde\rangle$. Repeating inductively the previous reasoning, we obtain that $\Mtilde$ has characteristic root~$r$, and in particular that $\Sigmatilde\neq\emptyset$. Hence $\M$ is not irreducible. 

Conversely, assuming that $\M$ is not irreducible, we prove that it does not have the spectral property. Let $\Sigma=\Sigma_1+\Sigma_2$ be a non trivial partition of $\Sigma$ such that $(a,b)\in I$ for all $(a,b)\in\Sigma_1\times\Sigma_2$, and thus $\M=\M_1\times\M_2$ with $\M_1=\langle\Sigma_1\rangle$ and $\M_2=\langle\Sigma_2\rangle$. The definition of the M\"obius polynomial shows that, with obvious notations: $\mu(z)=\mu_1(z)\mu_2(z)$ (this is a special case of Remark~\ref{rem:1}). Hence, if $r_1$ and $r_2$ denote respectively the characteristic roots of $\M_1$ and of~$\M_2$, one has $r=\min(r_1,r_2)$. Assume, say, that $r_2\geq r_1$. Pick any letter $a\in\Sigma_2$. Then $\M^a=\M_1\times\M^a_2$, and thus $r^a=\min(r_1,r^a_2)$, but $r^a_2\geq r_2\geq r_1$, hence $r^a=r$. This proves that $\M$ does not have the spectral property.
\end{proof}

\subsection{Linking sequences and linking executions}
\label{sec:compl-link-sequ}

We introduce the notion of linking execution, a technical tool for the proof of Theorem~\ref{thr:1} in the next section.

\begin{definition}
  \label{def:7}
  Let $(\M,X,\bot)$ be a concurrent system with $\M=\M(\Sigma,I)$ irreducible, let $\alpha\in X$ be a state. A \emph{linking sequence from~$\alpha$} is a sequence of letters $a_1,\dots,a_p$ satisfying, for some sequence of integers  $1\leq j_1<\dots<j_q\leq p$:
    \begin{enumerate}
    \item $\alpha\cdot (a_1\ldots a_p)\neq\bot$;
    \item $(a_{j_k},a_{j_{k+1}})\in D$ for $k=1,\dots,q-1$, where $D=(\Sigma\times\Sigma)\setminus I$;
    \item each letter of\/ $\Sigma$ has at least one occurrence in the sequence $(a_{j_1},\ldots,a_{j_q})$.
    \end{enumerate}
Let $a\in\Sigma$ be a letter. If the sequence of integers $j_1,\dots,j_q$ can be chosen such that $a_{j_1}=a$, we say that the linking sequence is \emph{$a$-rooted}.
    
    A \emph{linking execution from~$\alpha$} is an execution which is the image in $\M$ of a linking sequence from~$\alpha$. It is \emph{$a$-rooted} if it is the image in $\M$ of an $a$-rooted linking sequence.
  \end{definition}

Clearly, the existence of a linking sequence for a concurrent system $(\M,X,\bot)$ implies that the monoid $\M$ is irreducible.

\begin{proposition}
  \label{prop:2}
  Let $(\M,X,\bot)$ be an accessible concurrent system. Then the following statements are equivalent:
  \begin{enumerate}
  \item\label{item:3} $(\M,X,\bot)$ is irreducible.
  \item\label{item:4} For some state~$\alpha$ and for some letter~$a$, there exists an $a$-rooted linking execution from~$\alpha$.
  \item\label{item:5} For every state~$\alpha$ and for every letter~$a$, there exists an $a$-rooted linking execution from~$\alpha$.
  \end{enumerate}
\end{proposition}

\begin{proof}
  Assume that the concurrent system is irreducible.  Pick a state~$\alpha$ and a letter~$a$. We prove the existence of a $a$-rooted linking execution from~$\alpha$.  Let $b_1,\dots,b_q$ be a sequence of letters covering $\Sigma$ and such that $(b_j,b_{j+1})\in D$ for $j=1,\dots,q-1$ and $b_{j_1}=a$. Such a sequence exists since $\M$ is assumed to be irreducible. Then use the fact that the system is alive to decorate the sequence $(b_1,\dots,b_q)$ by inserting traces $x_i$ so that the trace $x=x_1 b_1 x_2b_2\ldots x_q b_q$ satisfies $\alpha\cdot x\neq\bot$. Then $x$ is an $a$-rooted linking execution from~$\alpha$, which proves the implication $\ref{item:3}\implies\ref{item:5}$. The implication $\ref{item:5}\implies\ref{item:4}$ is trivial. For the implication $\ref{item:4}\implies\ref{item:3}$, since the irreducibility of $\M$ is implied by the existence of a linking execution, it remains only to see that the system is alive. It   also follows from the existence of the $a$-rooted linking sequence, combined with the \accessibility{} of the system. The proof is complete.
\end{proof}

A key property of linking executions, that we will use in the proof of Theorem~\ref{thr:1} below, is the following result.

\begin{lemma}
  \label{lem:3}
  Let $a$ be a letter of a concurrent system $(\M,X,\bot)$, and let $x$ be an $a$-rooted linking execution from some state~$\alpha$. Let $\M^a=\langle\Sigma\setminus\{a\}\rangle$, fix $p$ and $q$ two integers and $\beta$ a state. Then the mapping:
  \begin{gather*}
\varphi:    \M_{\beta,\alpha}(p)\times\M^a_{\alpha\cdot x}(q)\to\M_{\beta},\qquad(u,v)\mapsto \varphi(u,v)=uxv
  \end{gather*}
  is injective.
\end{lemma}

\begin{proof}
  The only possibility for two pairs $(u,v)$ and $(u',v')$ to satisfy $uxv=u'xv'$ without having $u=u'$ and $v=v'$ would be that a letter from $v$ exchanges its location with a letter from $u$ by successive commutations. But this is impossible since this letter would have to cross the occurrences of $a$ in~$x$.
\end{proof}

\subsection{The spectral property for concurrent systems}
\label{sec:spectr-prop-conc}

The main result of this section is the following.

\begin{theorem}
  \label{thr:1}
  Let $\SS$ be a concurrent system that we assume to be accessible. Then $\SS$ is irreducible if and only if $\SS$ has the spectral property.
\end{theorem}

We first need a lemma.

\begin{lemma}
\label{lem:1}
  Let $(\M,X,\bot)$ be a concurrent system of characteristic root~$r$. Then there exists a state $\alpha$ and an integer $d$ such that, for every positive multiple $d'$ of~$d$, the series $\sum_{n\geq0}\card\M_{\alpha,\alpha}(nd')z^n$ has $r^{d'}$ as radius of convergence.
\end{lemma}

\begin{proof}
  Let $F$ be the adjacency matrix of $\ADSC$. According to Proposition~\ref{prop:7}, its spectral radius is $\rho(F)=1/r$. According to Proposition~\ref{prop:6}, there exists a node $u=(\alpha,c,i)$ and an integer $d>0$ such that, for any positive multiple $d'$ of~$d$, the series
  $$
  \widetilde Q(z)=\sum_{n\geq0} F^{d'n}_{u,u}z^n
  $$
  has $r^{d'}$ as radius of convergence. Fix $d'$ a positive multiple of~$d$, and let $Q(z)$ be the series
  $$
  Q(z)=\sum_{n\geq0}\card\M_{\alpha,\alpha}(nd')z^n.
  $$
  Let $R$ be the radius of convergence of~$Q(z)$; we prove that $R=r^{d'}$. Let the series:
  $$
U(z)=\sum_{n\geq0}\card\M_{\alpha,\alpha}(n)z^n.
$$
Its radius of convergence, say~$s$, satisfies $s\geq r$ by definition of~$r$. Hence, for any $|z|<r^{d'}$, the series $Q(z)$ is convergent, and thus $R\geq r^{d'}$. For the converse inequality, we observe that, for every integer $n>1$, according to Proposition~\ref{prop:5}:\quad $F^{d'n}_{u,u}\leq\card\M_{\alpha,\alpha}(d'n)$. Since $\widetilde Q(z)$ is rational with non negative coefficients and of radius of convergence~$r^{d'}$, the series $\widetilde Q(r^{d'})$ is divergent, and so $Q(r^{d'})$ is divergent as well. Hence $R\leq r^{d'}$, and finally $R=r^{d'}$, which was to be proved.
\end{proof}

\begin{proof}[Proof of Theorem~\ref{thr:1}.]
  Let $\SS=(\M,X,\bot)$ be an accessible concurrent system of characteristic root~$r$. The equivalence reduces to $\textsf{True}\iff\textsf{True}$ if the system is trivial, hence we assume in the remaining of the proof that the system is non trivial. It follows from Proposition~\ref{prop:3} that $r<+\infty$. 

  Assume that the system is irreducible. Let $a$ be an arbitrary letter of the base alphabet $\Sigma$ of the monoid. Let $\M^a=\langle\Sigma\setminus\{a\}\rangle$, and let $r^a$ be the characteristic root of $(\M^a,X,\bot)$. We prove that $r^a>r$.

Consider the concurrent system $(\M^a,X,\bot)$. We apply Lemma~\ref{lem:1} to obtain a state $\alpha_0$ and an integer $d>0$ such that, for any positive multiple $K$ of~$d$, the series
  \begin{gather}
    \label{eq:13}
    Q(z)=\sum_{n\geq0}t(n)z^n,\quad\text{with $t(n)=\card\M^a_{\alpha_0,\alpha_0}(nK)$,}
  \end{gather}
has $(r^a)^{K}$ as radius of convergence.

For this integer $d$, we claim that there exists a family of executions $(y_\alpha)_{\alpha\in X}$ and a positive integer $K$ with the following properties, valid for all $\alpha\in X$:
   \begin{enumerate}
   \item\label{item:6} $y_\alpha\in\M_{\alpha,\alpha}$ and $y_\alpha$ is an $a$-rooted linking execution;
   \item\label{item:8} $K$ is a multiple of~$d$, and $|y_\alpha|=K$.
   \end{enumerate}

To construct $(y_\alpha)_{\alpha\in X}$, we use Point~\ref{item:5} of Proposition~\ref{prop:2} to introduce first an $a$-rooted linking execution $u_\alpha$ starting from $\alpha$ for every state~$\alpha$. Since the concurrent system is \accessible{}, pick $v_\alpha\in\M_{\alpha\cdot u_\alpha,\alpha}$ and put $z_\alpha=u_\alpha v_\alpha$. Then $z_\alpha\in\M_{\alpha,\alpha}$, and $z_\alpha$ is still an $a$-rooted linking execution. Let $n_\alpha=|z_\alpha|$ and $m_\alpha=\prod_{\beta\in X,\ \beta\neq\alpha}n_\alpha$. By construction, $n_\alpha\geq|u_\alpha|>0$ for all states~$\alpha$, and thus $\beta_\alpha>0$ for all states~$\alpha$. Put finally $y_\alpha=(z_\alpha)^{d m_\alpha}$. Then $y_\alpha$ is still an $a$-rooted linking execution,  $y_\alpha\in\M_{\alpha,\alpha}$ and $|y_\alpha|=|z_\alpha|dm_\alpha$, which is a positive integer independent of $\alpha$ and multiple of~$d$, as required.

With this family $(y_\alpha)_{\alpha\in X}$ at hand, we construct inductively a family $(H^{\alpha,\beta}_n)_{n\geq0,(\alpha,\beta)\in X\times X}$ of sets  of executions by setting:
\begin{align}
  \label{eq:4}
  &H_{\alpha,\beta}(0)=\M^a_{\alpha,\beta}(0),\\
  \label{eq:7}
  \forall n\geq0\quad&
                                             H_{\alpha,\beta}({n+1})=\bigl(H_{\alpha,\beta}(n)\cdot y_\beta\bigr)\cup\bigcup_{\gamma\in X}H_{\alpha,\gamma}(n)\cdot\M^a_{\gamma,\beta}(K).
\end{align}
In~(\ref{eq:7}), we use the notations $U\cdot u$ and $U\cdot V$, for $U,V\subseteq\M$ and $u\in\M$ to denote $U\cdot u=\{x\cdot u\tqs x\in U\}$ and $U\cdot V=\{x\cdot y\tqs (u,v)\in U\times V\}$.

Using the symbols $+$ and $\sum$ to denote unions of pairwise disjoint families of sets, we claim that two following properties hold for all integers $n$ and for all $(\alpha,\beta)\in X\times X$:
\begin{align}
    \label{eq:5}
&  H_{\alpha,\beta}(n)\subseteq\M_{\alpha,\beta}(nK),\\
\label{eq:6}
 &     H_{\alpha,\beta}(n)=\sum_{i=0}^{n-1}\Bigl(\bigcup_{\gamma\in X}H_{\alpha,\gamma}({n-1-i})\cdot y_\gamma\cdot\M^a_{\gamma,\beta}(iK)\Bigr)+\M^a_{\alpha,\beta}\bigl(nK\bigr).
\end{align}

Property~(\ref{eq:5}) follows easily by induction on the integer~$n$, using in particular that $y_\beta\in\M_{\beta,\beta}(K)$, so we focus on~(\ref{eq:6}). Its is trivially true for $n=0$.

Assuming as an induction hypothesis that~(\ref{eq:6}) holds for $n\geq0$, we prove it for~$n+1$. Returning to~(\ref{eq:7}), we observe that the two terms in the right hand member are disjoint subsets of~$\M_{\alpha,\beta}$,  since $y_\beta$ is a $a$-rooted linking execution of length~$K$. Henceforth, using our disjoint union notation:
\begin{align}
  \label{eq:8}
  H_{\alpha,\beta}(n+1)&=\bigl(H_{\alpha,\beta}(n)\cdot y_\beta\bigr)+\bigcup_{\gamma\in X}H_{\alpha,\gamma}(n)\cdot\M^a_{\gamma,\beta}(K).
\end{align}
We replace $H_{\alpha,\gamma}(n)$ in~(\ref{eq:8}) by its expression given by the induction hypothesis, and distribute the union to get:
  \begin{multline}
\label{eq:10}
    H_{\alpha,\beta}(n+1)=\bigl(H_{\alpha,\beta}(n)\cdot y_\beta\bigr)\\
                              +\sum_{i=0}^{n-1}\bigcup_{\delta\in X} H_{\alpha,\delta}(n-1-i)\cdot y_\delta\cdot
                            \Bigl(\bigcup_{\gamma\in X} \M^a_{\delta,\gamma}(iK)\cdot\M^a_{\gamma,\beta}(K)\Bigr)\\
                             +\bigcup_{\gamma\in X}\M^a_{\alpha,\gamma}(nK)\cdot\M^a_{\gamma,\beta}(K)
                        \end{multline}

Observe that the following equality is true in any concurrent system for every integers $p,q\geq0$ and states $\alpha,\beta$:
\begin{gather}
\label{eq:9}
  \M_{\alpha,\beta}(p+q)=\bigcup_{\gamma\in X}\M_{\alpha,\gamma}(p)\cdot\M_{\gamma,\delta}(q).
\end{gather}

We use~(\ref{eq:9}) twice to transform~(\ref{eq:10}) and then rearrange the terms in the sum and obtain:
\begin{align*}
\label{eq:11}
    H_{\alpha,\beta}(n+1)&=\bigl(H_{\alpha,\beta}(n)\cdot y_\beta\bigr)
                           +\sum_{i=1}^{n}\Bigl(\bigcup_{\delta\in X}H_{\alpha,\delta}(n-i)\cdot y_\delta\cdot\M^a_{\delta,\beta}(iK)\Bigr)\\
                         &\qquad+\M^a_{\alpha,\beta}\bigl((n+1)K\bigr)\\
&=\sum_{i=0}^n\Bigl(\bigcup_{\gamma\in X}
H_{\alpha,\gamma}(n-i)\cdot y_\gamma\cdot\M^a_{\gamma,\beta}(iK)\Bigr)+\M^a_{\alpha,\beta}\bigl((n+1)K\bigr).
\end{align*}
This completes the induction and proves~(\ref{eq:6}).

We now consider the following generating series:
\begin{gather*}
  R_{\alpha,\beta}(z)=\sum_{n\geq0}\lambda_{\alpha,\beta}(n)z^n,\quad\text{with $\lambda_{\alpha,\beta}(n)=\card H_{\alpha,\beta}(n)$,}
\end{gather*}
of radius of convergence $\rho_{\alpha,\beta}$. Recalling that $r$ denotes the minimum of all radius of convergence $r_{\alpha,\beta}$ of all growth series~$G_{\alpha,\beta}(z)$, one has for every non negative real $t<r^{K}$:
\begin{gather*}
  R_{\alpha,\beta}(t)\leq\sum_{n\geq0}\#\M_{\alpha,\beta}(nK)(t^{\frac1K})^{nK}<+\infty,
\end{gather*}
therefore:
\begin{gather}
  \label{eq:12}
\forall(\alpha,\beta)\in X\times X\quad  \rho_{\alpha,\beta}\geq r^K.
\end{gather}

We shall now relate the two sequences $(\lambda_{\alpha_0,\alpha_0}(n))_{n\geq0}$ and $(t_n)_{n\geq0}$, where the latter has been defined in~(\ref{eq:13}). Based on~(\ref{eq:6}), one has:
\begin{align}
 \label{eq:14} \lambda_{\alpha_0,\alpha_0}(n+1)&\geq\sum_{i=0}^n\Bigl(\lambda_{\alpha_0,\alpha_0}(n-i)t(i)\Bigr)+t(n+1).
\end{align}

To establish~(\ref{eq:14}), we have used the following fact: $\card\bigl(H_{\alpha,\gamma}(n)\cdot y_\gamma\cdot\M^a_{\gamma,\beta}(iK)\bigr)=\bigl(\card H_{\alpha,\gamma}(n)\bigr)\cdot\bigl(\card
\M^a_{\gamma,\beta}(iK)\bigr)$, which holds according to Lemma~\ref{lem:3} since $y_\gamma$ is an $a$-rooted linking execution. We recognize in the right hand member of~(\ref{eq:14}) the coefficient of a product series. Hence, multiplying by $z^n$ and summing over $n\geq0$ yields, for every non negative real $z$ where the series converge:
\begin{align*}
\frac1z\bigl(R_{\alpha_0,\alpha_0}(z)-1\bigr)\geq
R_{\alpha_0,\alpha_0}(z)Q(z)+\frac1z\bigl(Q(z)-1),
\end{align*}
and thus:
\begin{align*}
  Q(z)\leq\frac{R_{\alpha_0,\alpha_0}(z)}{1+zR_{\alpha_0,\alpha_0}(z)}.
\end{align*}
Henceforth, $Q(z)$ is bounded on the interval $(0,\rho_{\alpha_0,\alpha_0})$, hence on the interval $(0,r^K)$ thanks to~(\ref{eq:12}). But $Q(z)$ is a rational series with non negative coefficients. Hence, thanks to Remark~\ref{rem:2}, its radius of convergence $(r^a)^K$ is one of its pole, and thus $(r^a)^K>r^K$ and finally $r^a>r$, which was to be proved.

\medskip
For the converse part, we assume that the system is not irreducible and we prove that it does not have the spectral property. Since the system is accessible, two cases may occur.

\emph{First case: the monoid $\M$ is not irreducible}.\quad We follow the same line of proof as in the proof of Proposition~\ref{prop:1}. Let thus $\Sigma=\Sigma_1+\Sigma_2$ be a non trivial partition of~$\Sigma$, such that $\Sigma_1\times\Sigma_2\subseteq I$ and thus $\M=\M_1\times\M_2$ with $\M_1=\langle\Sigma_1\rangle$ and $\M_2=\langle\Sigma_2\rangle$. Then $\mu(z)=\mu_1(z)\mu_2(z)$ according to Remark~\ref{rem:1}, where $\mu_1(z)$ and $\mu_2(z)$ denote the M\"obius matrices of $(\M_1,X,\bot)$ and of $(\M_2,X,\bot)$, of characteristic roots $r_1$ and~$r_2$. It follows that $r=\min(r_1,r_2)$. Assume, say, that $r_2\geq r_1$. Pick any letter $a\in\Sigma_2$. Then obvious comparisons on growth series show that, with obvious notations: $r^a_2\geq r_2$. Since $\M^a=\M_1\times\M^a_2$, one also have $\mu^a(z)=\mu_1(z)\mu_2^a(z)$ and thus $r^a=\min(r_1,r^a_2)=r$. Hence $(\M,X,\bot)$ does not have the spectral property.

\emph{Second case: the system is not alive}. There exists a state $\alpha_0$ and a letter $a$ such that $\M_{\alpha_0}=\M_{\alpha_0}^a$. But since the system is \accessible, it implies that $\M_\alpha=\M^a_\alpha$ for every state~$\alpha$. Hence $r=r^a$ and $(\M,X,\bot)$ does not have the spectral property.

This completes the proof of the theorem.
\end{proof}


\begin{corollary}
\label{thr:2}
  Let $(\M,X,\bot)$ be a non trivial and irreducible concurrent system of characteristic root~$r$. Then:
 $\ADSC^+$ has the same spectral radius $r^{-1}$ as the spectral radius of\/ $\ADSC$.
\end{corollary}

\begin{proof}
  Let $F$ be the adjacency matrix of $\ADSC$, and let $\rho=\rho(F)$. Then $\rho=r^{-1}$ according to Proposition~\ref{prop:7}. After a simultaneous permutation of the lines and columns of $F$ in order to put the positive nodes $(\alpha,c,i)$ of $\ADSC$ in first position, the matrix $F$ has the following form:
  \begin{gather*}
    F=\begin{pmatrix}
      F^+&J\\
      0&F^0
    \end{pmatrix}
  \end{gather*}
  where $F^0$ is the adjacency matrix of the digraph $\ADSC^0=\ADSC\setminus\ADSC^+$. Indeed, according to Proposition~\ref{lem:2}, point~\ref{item:7}, null nodes do not lead to positive nodes in~$\DSC$, whence the $0$-matrix on the left of~$F^0$. 

  Hence $\rho(F)=\max(\rho^+,\rho^0)$ on the one hand, where $\rho^+=\rho(F^+)$ and $\rho^0=\rho(F^0)$; and $\rho(F)=r^{-1}$ according to Proposition~\ref{prop:7} on the other hand.

    It follows from point~\ref{item:10} of Proposition~\ref{lem:2} that executions of $(\M,X,\bot)$, the normal form of which start with a null node, belong to $\bigcup_{a\in\Sigma}\M^a$. Hence $\rho^0<\rho$ since the concurrent system $(\M,X,\bot)$ satisfies the spectral property according to Theorem~\ref{thr:1}, and thus $\rho=\rho^+$, which was to be proved.
%
%
\end{proof}

\section{Probabilistic applications of the spectral property of concurrent systems}
\label{sec:appl-spectr-prop}

\subsection{Boundary at infinity and uniform measures}
\label{sec:boundary-at-infinity}

\subsubsection{Boundary at infinity of a trace monoid}
\label{sec:boundary-at-infinity-1}

Let $\M$ be a trace monoid. The \emph{boundary at infinity} of $\M$ is the topological space $\BM=\bigl\{(c_i)_{i\geq1}\tqs c_i\in\Cstar,\quad c_i\to c_{i+1}\quad\forall i\geq1\bigr\}$, with the topology induced by the product topology on~$\Cstar^{\bbZ_{\geq1}}$.

This \emph{ad hoc} construction is a short way for obtaining a compactification $\Mbar=\M\cup\BM$ of~$\M$. For each non empty trace $x\in\M$, of height $h=\height(x)$, let $(c_1,\dots,c_h)$ be the normal form of~$x$. We define $C_i:\M\to\C$ by $C_i(x)=c_i$ if $i\leq h$ and by $C_i(x)=\vd$ for $i>h$, and $c_i(\vd)=\vd$ for all integers $i\geq1$. We obtain thus a family $(C_i)_{i\geq1}$ of mappings $C_i:\Mbar\to\C$, where $(C_i)_{i\geq1}$ is defined on $\BM$ as the family of natural projections. We have in particular $\xi\in\BM$ if and only if $C_i(\xi)\neq\vd$ for all $i\geq1$.

Equip $\M$ with the partial order $\leq$ inherited from the monoid multiplication, defined by $x\leq y\iff\exists z\in\M\quad y=xz$. This partial order is extended on $\Mbar$ by setting:
\begin{gather*}
  \forall \xi,\xi'\in\Mbar\quad\xi\leq\xi'\iff\bigl(\forall i\geq1\quad C_i(\xi)\leq C_i(\xi')\bigr).
\end{gather*}

The \emph{visual cylinder} of base $x\in\M$ is $\up x=\{\xi\in\BM\tqs x\leq \xi\}$. Let $\FFF$ be the Borel \slgb\ on~$\BM$. Then the family $\{\emptyset\}\cup\{\up x\tqs x\in\M\}$ is a $\pi$-system that generates~$\FFF$. In particular, every probability measure $\nu$ on $(\BM,\FFF)$ is entirely determined by its values $\nu(\up x)$ on visual cylinders.

\medskip
Consider now a concurrent system $(\M,X,\bot)$ and an initial state~$\alpha$. Let $\overline{\M_\alpha}$ be the topological closure of $\M_\alpha$ in~$\Mbar$. We put: $\BM_\alpha=\Mbar_\alpha\cap\BM$. Elements of $\BM_\alpha$ correspond thus to ``infinite executions starting from~$\alpha$''.

When considering a probability measure $\nu$ on~$\BM_\alpha$, one might equivalently consider $\nu$ as a probability measure on $\BM$ such that $\nu(\up x)=0$ for all $x\notin\M_\alpha$. In particular, $\nu$~is entirely determined by its values $\nu(\up x)$ on elementary cylinders $\up x$ for $x\in\M_\alpha$.

\subsubsection{Existence of a uniform measure}
\label{sec:exist-unif-meas}

Following~\cite{abbes19:_markov}, we introduce the notion of uniform measure. 
%

\begin{definition}
  \label{def:8}
  
  Let $(\M,X,\bot)$ be a concurrent system. A \emph{uniform measure} is a family $(\nu_\alpha)_{\alpha\in X}$ such that $\nu_\alpha$ is a probability measure on $\BM_\alpha$ for every $\alpha\in X$, satisfying the chain condition: 
    \begin{gather}
      \label{eq:24}
    \forall\alpha\in X\quad\forall x\in\M_\alpha\quad\forall y\in\M_{\alpha\cdot x}\quad\nu_\alpha\bigl(\up(xy)\bigr)=\nu_\alpha(\up x)\nu_{\alpha\cdot x}(\up y)\:,
  \end{gather}
and such that, for some positive function $\Gamma:X\times X\to\bbR_{>0}$ and for some positive real~$t$:
  \begin{gather}
    \label{eq:17}
    \forall\alpha\in X\quad\forall x\in\M_\alpha\quad \nu_\alpha(\up x)=t^{|x|}\Gamma(\alpha,\alpha\cdot x).
  \end{gather}
If\/ $(\nu_\alpha)_{\alpha\in X}$ is a uniform measure, its \emph{induced fibred valuation} is $f=(f_\alpha)_{\alpha\in X}$ defined by $f_\alpha(x)=\nu_\alpha(\up x)$ for all $\alpha\in X$ and $x\in\M_\alpha$.
\end{definition}

Let $X$ be a set. A \emph{cocycle on $X$} is a positive function $\Gamma:X\times X\to \bbR_{>0}$ such that:\quad$\forall(\alpha,\beta,\gamma)\in X\times X\times X\quad\Gamma(\alpha,\gamma)=\Gamma(\alpha,\beta)\Gamma(\beta,\gamma)$.

Assume that $(\M,X,\bot)$ is accessible and that $f=(f_\alpha)_{\alpha\in X}$ is the induced fibred valuation of a uniform measure as in~(\ref{eq:17}). Then $\Gamma$ must be a cocycle. Indeed, let $\alpha,\beta,\gamma\in X$, and let $x\in\M_\alpha$ and $y\in\M_{\alpha\cdot x}$ be such that $\beta=\alpha\cdot x$ and $\gamma=\beta\cdot y$. Then evaluating $f_\alpha(xy)$ through~(\ref{eq:24}) on the one hand, and through~(\ref{eq:17}) on the other hand, yields: $t^{|xy|}\Gamma(\alpha,\gamma)=t^{|x|}\Gamma(\alpha,\beta)t^{|y|}\Gamma(\beta,\gamma)$, whence the sought relation $\Gamma(\alpha,\gamma)=\Gamma(\alpha,\beta)\Gamma(\beta,\gamma)$.


\medskip
The existence of a uniform measure for a concurrent system $(\M,X,\bot)$, follows from the following construction, inspired by the Patterson-Sullivan construction (see, \eg,~\cite[Th.~5.4]{coornaert93}) and detailed in~\cite{abbes19:_markov}. Let $r$ be the characteristic root of the concurrent system. For each state $\alpha\in X$ and for each real $t\in(0,r)$, let $\nu_{\alpha,t}$ be the discrete probability measure on $\M\subseteq\Mbar$ defined by:
\begin{gather}
  \label{eq:34}
  \nu_{\alpha,t}=\frac1{G_{\alpha}(t)}\sum_{x\in\M_\alpha}\delta_{\{x\}}t^{|x|}\,,
\end{gather}
where $\delta_{\{x\}}$ denotes the Dirac measure on~$x$. Then, for each $\alpha\in X$, the family $(\nu_{\alpha,t})_{t\in(0,r)}$ converges weakly, as $t\to r$, toward a probability measure $\nu_\alpha$ on $\BM_\alpha$ such that $\nu=(\nu_\alpha)_{\alpha\in X}$ is a uniform measure. The associated cocycle is the \emph{Parry cocycle}, given by:
\begin{gather}
  \label{eq:29}
    \forall(\alpha,\beta)\in X\times X\quad \Gamma(\alpha,\beta)=\lim_{\substack{t\to r\\t<r}}\frac{G_\beta(t)}{G_\alpha(t)}\in(0,+\infty),
  \end{gather}
and one has $\nu_\alpha(\up x)=r^{|x|}\Gamma(\alpha,\alpha\cdot x)$ for all $\alpha\in X$ and for all $x\in\M_\alpha$.

\medskip

The uniqueness of the uniform measure was a question left open in \cite{abbes19:_markov}. We prove it below in Section~\ref{sec:uniq-unif-meas} for irreducible concurrent systems. 

\subsubsection{Markov chain of states-and-cliques}
\label{sec:markov-chain-states}

Throughout this section~\ref{sec:markov-chain-states}, we consider a uniform measure $\nu=(\nu_\alpha)_{\alpha\in X}$ on a concurrent system $(\M,X\bot)$. For each state~$\alpha$, the probability measure $\nu_\alpha$ is entierely characterized by its values $\nu_\alpha(\up x)$ on visual cylinders~$\up x$, for $x$ ranging over~$\M_\alpha$. The construction of the boundary at infinity $\BM_\alpha$ would however appeal for the values of $\nu_\alpha$ on the ``standard cylinders'' $\{C_1=x_1,\ldots,C_k=x_k\}$ for $k\in\bbZ_{\geq1}$ and $(x_1,\ldots,x_k)\in\Cstar^k$. Put differently, the question is to determine the nature of the probabilitic process $(C_k(\xi))_{k\geq1}$ when $\xi$ is an infinite execution drawn at random according to the probability measure~$\nu_\alpha$.

A complete answer, that we recall now, is given in~\cite{abbes19:_markov} when considering the process of states-and-cliques rather than the process of cliques only. For every initial state~$\alpha$, and for every $\xi\in\BM_\alpha$, let $(\alpha_i)_{i\geq0}$ be the sequence of states encountered by the infinite execution $\xi$ at the successive stages of its normal form. That is to say, $(\alpha_i)_{i\geq0}$ is defined by $\alpha_i(\xi)=\alpha\cdot\bigl(C_1(\xi)\cdot\ldots\cdot C_i(\xi)\bigr)$ for $i\geq0$.

Then, with respect to the probability measure~$\nu_\alpha$, the sequence $(\alpha_i,C_{i+1})_{i\geq0}$ is a Markov chain~\cite[Th.~4.5]{abbes19:_markov}, called the \emph{Markov chain of states-and-cliques} (\MCSC).

Let $f=(f_\alpha)_{\alpha\in X}$ be the fibred valuation induced by the uniform measure~$\nu$ (see Definition~\ref{def:8}). The initial measure and the transition kernel (or matrix) of the $\MCSC$ can be related to $f$ through the M\"obius transform\footnote{If $\varphi:\C\to\bbR$ is a real-valued function, the \emph{M\"obius transform} \cite{rota64} of $\varphi$ is the function $h:\C\to\bbR$ defined by:
\begin{gather*}
  \forall c\in\C\quad h(c)=\sum_{c'\in\C\tqs c\leq c'}(-1)^{|c'|-|c|}\varphi (c').
\end{gather*}
Here we assume that~$\varphi$, and thus~$h$, is only defined on~$\C$. An extension of $h$ to $\M$ is also possible if $\varphi$ is defined on~$\M$, and relevant from the probabilistic point of view; see \cite{abbes19:_markov} for this extension, that we shall not need in this paper.
\\
\strut\quad Note that the restricted partial order $(\C,\leq)$ corresponds to the inclusion order on cliques seen as subsets of~$\Sigma$. Furthermore, it $\varphi:\C\to\bbR$ is $\varphi(c)=z^{|c|}$, then $h(\vd)=\mu(z)$, the M\"obius polynomial of~$\M$.} of each~$f_\alpha$, as we explain now.

For every initial state~$\alpha\in X$, the initial distribution of the \MCSC\ is $\delta_{\{\alpha\}}\otimes h_\alpha$, where $h_\alpha$ is the M\"obius transform of~$f_\alpha$. In other words, when taking at random under $\nu_\alpha$ an infinite execution $\xi$ starting from~$\alpha$, the law of its first clique $C_1(\xi)$ is given by~$h_\alpha$.

For any state $\alpha\in X$, let $g_\alpha:\Cstar\to\bbR_{\geq0}$ be the function defined by:
\begin{gather}
  \label{eq:26}
\forall c\in\Cstar\quad  g_\alpha(c)=\sum_{d\in\Cstar_\beta,\ c\to d}h_\beta(d),\quad\text{where $\beta=\alpha\cdot c$.}
\end{gather}

Then the transition matrix $M$ of the $\MCSC$ is independent of~$\alpha$, and given by~\cite[Th.~4.5]{abbes19:_markov}:
\begin{gather}
  \label{eq:30}
  M_{(\alpha,c),(\beta,d)}=\un(\beta=\alpha\cdot c)\un(c\to d)\frac{h_{\beta}(d)}{g_\alpha(c)},\quad\text{if $g_\alpha(c)\neq0$.}
\end{gather}

We shall need additional informations on the M\"obius transforms~$h_\alpha$. The two following properties, proved in~\cite[Th.~4.5]{abbes19:_markov}, can be seen as normalization conditions:
\begin{gather}
  \label{eq:27}
  \forall\alpha\in X\quad  h_\alpha(\vd)=0,\\
  \forall\alpha\in X\quad\forall c\in\Cstar_\alpha\quad h_\alpha(c)\geq0.
\end{gather}

Furthermore, assume that $f=(f_\alpha)_{\alpha\in X}$ is any family of real valed functions $f_\alpha:\M\to\bbR$, and let $(h_\alpha)_{\alpha\in X}$ be the corresponding M\"obius transforms. Assume that the chain relations $f_\alpha( xy)=f_\alpha(x)f_{\alpha\cdot x}(y)$ hold for all $x\in\M_\alpha$ and $y\in\M_{\alpha\cdot x}$, and define $(g_\alpha)_{\alpha\in X}$ as in~(\ref{eq:26}). Then~(\ref{eq:27}) only implies the following identities~\cite[Lemma~4.7]{abbes19:_markov}:
\begin{gather}
  \label{eq:25}
  \forall\alpha\in X\quad\forall c\in\Cstar_\alpha\quad
  h_\alpha(c)=f_\alpha(c)g_\alpha(c)
\end{gather}

Hence~(\ref{eq:25}) holds in particular for $f=(f_\alpha)_{\alpha\in X}$, the fibred valuation associated to the uniform measure. Therefore, the nodes $(\alpha,c)$ such that $g_\alpha(c)=0$ correspond also to those nodes such that $h_\alpha(c)=0$. In view of the form~(\ref{eq:30}) of the transition matrix~$M$ of the initial distribution of the \MCSC, they are therefore not reached by the \MCSC, hence the restriction $g_\alpha(c)\neq0$ in~(\ref{eq:30}) is of no matter.

\subsection{Uniqueness of the uniform measure}
\label{sec:uniq-unif-meas}

Let $\nu=(\nu_\alpha)_{\alpha\in X}$ be a uniform measure of a concurrent system $(\M,X,\bot)$, given by $\nu_\alpha(\up x)=s^{|x|}\Delta(\alpha,\alpha\cdot x)$ for some positive real $s$ and some cocycle $\Delta:X\times X\to\bbR_{>0}$. Our aim is to prove that $s=r$ and $\Delta=\Gamma$ where $r$ is the characteristic root of the system and $\Gamma$ is the Parry cocycle given in~(\ref{eq:29}), under the hypothesis that the system is irreducible.

The following observation relates $\Delta$ and the M\"obius matrix.

\begin{proposition}
  \label{prop:9}
  Let $(\M,X,\bot)$ be a concurrent system, and assume that $\nu_\alpha(\up x)=s^{|x|}\Delta(\alpha,\alpha\cdot x)$ defines a uniform measure.

Let $\mu=\mu(s)$ be the M\"obius matrix of the system evaluated at~$s$. Then: for any arbitrary state $\alpha_0\in X$, the positive vector $u=(u_\alpha)_{\alpha\in X}$ defined by $u_\alpha=\Delta(\alpha_0,\alpha)$ satisfies $u\in\ker \mu$.
\end{proposition}

\begin{proof}
  Let $f=(f_\alpha)_{\alpha\in X}$ be the fibred valuation induced by $\nu=(\nu_\alpha)_{\alpha\in X}$, and let $h_\alpha$ be the M\"obius transform of $f_\alpha$ for each $\alpha\in X$. We write down the identity $h_\alpha(\vd)=0$ from~(\ref{eq:27}), which yields on the one hand:
  \begin{gather}
    \label{eq:33}
    \sum_{c\in\C_\alpha}(-1)^{|c|}s^{|c|}\Delta(\alpha,\alpha\cdot c)=0.
  \end{gather}

On the other hand, the vector $v=\mu\cdot u$ evaluates as follows:
  \begin{align*}
    v_\alpha&=\sum_{\beta\in X}\Delta(\alpha_0,\beta)\Bigl(\sum_{c\in\C_{\alpha,\beta}}(-1)^{|c|}s^{ |c|}\Bigr).
  \end{align*}
  Writing $\Delta(\alpha_0,\beta)=\Delta(\alpha_0,\alpha)\Delta(\alpha,\beta)$, the above expression is thus proportional to the left member of~(\ref{eq:33}), which vanishes.
\end{proof}

Our next observation is the following simple result.

\begin{lemma}
  \label{lem:5}
  Let $(\M,X,\bot)$ be a concurrent system. Assume that $f=(f_\alpha)_{\alpha\in X}$ is the fibred valuation induced by some uniform measure~$\nu$, and let $h_\alpha$ denote the M\"obius transform of~$f_\alpha$. Then $h_\alpha(c)>0$ for every positive node $(\alpha,c)\in\DSC^+$.
\end{lemma}

\begin{proof}
Let $(\alpha,c)$ be a positive node, and let $x\in\M_\alpha$ be an $(\alpha,c)$-protection. With the language of infinite executions, one has: $\{\xi\in\BM_\alpha\tqs C_1(\xi)=c\}\supseteq\up x$. Therefore, $\nu_\alpha(C_1=c)\geq\nu_\alpha(\up x)>0$. But $\nu_\alpha(C_1=c)=h_\alpha(c)$ since the initial measure of the $\MCSC$ with initial state $\alpha$ is $\delta_{\{\alpha\}}\otimes h_\alpha$, as recalled in Section~\ref{sec:markov-chain-states}. Therefore $h_\alpha(c)>0$.  
\end{proof}

A key lemma is now the following.

\begin{lemma}
  \label{lem:4}
  Let $(\M,X,\bot)$ be a non trivial irreducible concurrent system, and let $F$ be the adjacency matrix of the corresponding\/ $\ADSC^+$. Assume that $\nu=(\nu_\alpha)_{\alpha\in X}$ is a uniform measure with induced fibred valuation $f=(f_\alpha)_{\alpha\in X}$ given by $f_\alpha(x)=s^{|x|}\Delta(\alpha,\alpha\cdot x)$ for all $\alpha\in X$ and $x\in\M_\alpha$, for some positive real $s$ and some cocycle~$\Delta$. Let $h_\alpha$ denote the M\"obius transform of~$f_\alpha$. Fix an arbitrary state $\alpha_0\in X$, and let $u$ be the vector defined by:
  \begin{gather}
    \label{eq:35}
    u(\alpha,c,i)=\frac1{s^{i-1}}\Delta(\alpha_0,\alpha)h_\alpha(c),
  \end{gather}
for all $\alpha\in X$, $c\in\Cstar_\alpha$ and $i\in\{1,\ldots,|c|\}$ such that $(\alpha,c,i)\in\ADSC^+$. Then $s^{-1}$ is an eigenvalue of\/~$F$ for which  $u$ is a positive right eigenvector.
\end{lemma}

\begin{proof}
  Let $J$ denote the set of nodes of~$\ADSC^+$. Lemma~\ref{lem:5} implies that $h_\alpha(c)>0$ for every $(\alpha,c)\in\DSC^+$. Hence $u$ is a positive vector.

  We prove that $Fu=(1/s)u$, which will prove the remaining of the statement. For any $(\alpha,c,i)\in J$ such that $i<|c|$, the row $F_{(\alpha,c,i),\bullet}$ is identically zero, except for the entry of the column indexed by $(\alpha,c,i+1)$. Therefore the identity $(Fu)_{(\alpha,c,i)}=(1/s)u(\alpha,c,i)$ is obvious.

  We consider for each $\alpha\in X$ the function $g_\alpha:\Cstar\to\bbR_{>0}$ defined as in~(\ref{eq:26}), and then the identity~(\ref{eq:25}) holds. Therefore we compute as follows for $i=|c|$, putting $\beta=\alpha\cdot c$\,:
\begin{align*}
  (Fu)_{(\alpha,c,|c|)}&=\sum_{d\in\Cstar_\beta\tqs c\to d}u_{(\beta,d,1)}\\
                      &=\Delta(\alpha_0,\beta)g_\alpha(c)\\
                      &=\Delta(\alpha_0,\beta)\frac{h_\alpha(c)}{s^{|c|}\Delta(\alpha,\beta)}&&
                                                                                               \text{using~(\ref{eq:25})}\\
                      &=\frac 1s\Delta(\alpha_0,\alpha)\frac{h_\alpha(c)}{s^{|c|-1}}&&\text{using the cocycle property of $\Delta$}\\
  &=\frac1su(\alpha,c,|c|).
\end{align*}
The proof is complete.
\end{proof}

This implies at once, \emph{via} strong results on non negative reducible matrices, the following property of~$\ADSC^+$.

\begin{proposition}
  \label{prop:10}
  Let $(\M,X,\bot)$ be an irreducible concurrent system of characteristic root~$r$. Then, among the strongly connected components of\/ $\ADSC^+$, those of spectral radius $r^{-1}$ are exactly the terminal components.
\end{proposition}

\begin{proof}
  Let $F$ be the adjacency matrix of~$\ADSC^+$. Using the existence of a uniform measure associated with the characteristic root~$r$ on the one hand, and Lemma~\ref{lem:4} on the other hand, we obtain the existence of a positive $r^{-1}$-eigenvector of~$F$. Since the spectral radius of $\ADSC^+$ is $r^{-1}$ according to Corollary~\ref{thr:2}, the result follows according to \cite[Fact 12(b)]{rothblum14}.
\end{proof}

\begin{remark}
  The reference \cite{rothblum14} in the above proof reveals a deep connection with the structure of non negative \emph{reducible} matrices. These algebraic aspects are further discussed in Section~\ref{sec:an-alternative-point}.
\end{remark}

\begin{theorem}
  \label{thr:4}
  Let $(\M,X,\bot)$ be an irreducible concurrent system. Then there exists a unique uniform measure $\nu=(\nu_\alpha)_{\alpha\in X}$ associated to the concurrent system. This uniform measure is entirely characterized by:
  \begin{gather*}
    \forall \alpha\in X\quad\forall x\in\M_\alpha\quad\nu_\alpha(\up x)=r^{|x|}\Gamma(\alpha,\alpha\cdot x),
  \end{gather*}
where $r$ is the characteristic root of the concurrent system, and\/ $\Gamma:X\times X\to\bbR_{>0}$ is the Parry cocycle introduced in~{\normalfont(\ref{eq:29})}.
\end{theorem}

\begin{proof}
  The existence part was the topic of Section~\ref{sec:exist-unif-meas}, hence we focus on proving the uniqueness. Let $(\nu_\alpha)_{\alpha\in X}$ be a uniform measure and let $f=(f_\alpha)_{\alpha\in X}$ be the induced fibred valuation. Let $s>0$ and $\Delta:X\times X\to\bbR_{>0}$ be the cocycle such that:
  \begin{gather}
    \label{eq:28}
    \forall \alpha\in X\quad\forall x\in\M_\alpha\quad f_\alpha(x)=s^{|x|}\Delta(\alpha,\alpha\cdot x).
  \end{gather}
  For each $\alpha\in X$, let $h_\alpha:\C\to\bbR$ be the M\"obius transform of~$f_\alpha$, and let $g_\alpha$ be defined as in~(\ref{eq:26}) with respect to~$h_\alpha$.

  We first prove that $s=r$. Fix an arbitrary state $\alpha_0\in X$. Let $J$ be the set of nodes of $\ADSC$, and let $u:J\to\bbR$ be defined as in Lemma~\ref{lem:4}. Then $u$ is a positive $s^{-1}$-eigenvector of~$F$. Since $F$ has spectral radius $r^{-1}$ according to Corollary~\ref{thr:2}, it follows from \cite[Fact 6a]{rothblum14} that $s^{-1}=r^{-1}$, and thus $s=r$.

It remains only to prove that $\Delta=\Gamma$, where $\Gamma$ is the Parry cocycle. Consider a terminal component $T$ of $\ADSC^+$. It corresponds in the obvious way to a terminal component $\Ttilde$ in~$\DSC^+$.
 Let $N$ be the set of nodes of $\DSC^+$ belonging to~$\Ttilde$. Let $\nu'=(\nu'_\alpha)_{\alpha\in X}$ be the uniform measure associated with $f'_\alpha(x)=r^{|x|}\Gamma(\alpha,\alpha\cdot x)$, defined for $\alpha\in X$ and $x\in\M_\alpha$, where $\Gamma$ is the Parry cocyle. Both uniform measures, $\nu$~and~$\nu'$, give raise to a Markov chain of states-and-cliques on the nodes of $\DSC$. We claim that:
\begin{intermediate}{$(\dag)$}
  The transition matrices of these two Markov chains are equal on $N\times N$.
\end{intermediate}

Since the state $\alpha_0$ was arbitrary, we assume without loosing generality that it has be chosen in such a way that $(\alpha_0,c)\in N$ for at least some clique~$c$. Now consider the vector $u'$ defined as $u$ was defined, but relatively to the uniform measure~$\nu'$. For the same reasons as for~$u$, the restriction of $u'$ to $T$ is a Perron eigenvector of the adjacency matrix of~$T$, which is irreducible. Henceforth $u$ and $u'$ are proportional on the nodes of~$T$.

Let $h'_\alpha$ denote the M\"obius transform of~$f'_\alpha$. For some positive constant~$k$, one has thus:
\begin{gather}
  \label{eq:32}
  \forall (\alpha,c)\in N\quad \Delta(\alpha_0,\alpha)h_\alpha(c)=k\Gamma(\alpha_0,\alpha)h'_\alpha(c).
\end{gather}
It yields in particular, using the cocycle identities $\Delta(\alpha_0,\alpha_0)=\Gamma(\alpha_0,\alpha_0)=1$\,:
\begin{gather}
  \label{eq:36}
  \forall c\in\Cstar_{\alpha_0}\quad (\alpha_0,c)\in N\implies h_{\alpha_0}(c)=kh'_{\alpha_0}(c).
\end{gather}

Let $M$ and $M'$ be the transition matrices of the \MCSC\ associated with~$\nu$ and with~$\nu'$. Recalling the identity $h_\alpha=f_\alpha g_\alpha$ from~(\ref{eq:25}), we compute according to~(\ref{eq:30}):
\begin{align}
  \notag
  M_{(\alpha_0,c),(\beta,d)}&=\un(\beta=\alpha_0\cdot c)\un(c\to d)\frac{h_\beta(d)}{g_{\alpha_0(c)}}\\
\notag                            &=\un(\beta=\alpha_0\cdot c)\un(c\to d)r^{|c|}\Delta(\alpha_0,\beta)\frac{h_\beta(d)}{h_{\alpha_0}(c)}\\
\label{eq:31}\tag{*}  &=\un(\beta=\alpha_0\cdot c)\un(c\to d)r^{|c|}\Gamma(\alpha_0,\beta)\frac{h'_\beta(d)}{h'_{\alpha_0}(c)}\\
  \notag&=M'_{(\alpha_0,c),(\beta,d)}
\end{align}
where we have used both~(\ref{eq:32}) and~(\ref{eq:36}) in the line~(\ref{eq:31}). This proves that, in the transition matrices, the two lines corresponding to the node $(\alpha_0,c)$ are equal. But since $\alpha_0$ was arbitrarily chosen such that $(\alpha_0,c)\in N$, this proves the claim~$(\dag)$.

\medskip
We now complete the proof of the equality $\Delta=\Gamma$. Fix $(\alpha_0,c)\in N$. Let $z$ be an $(\alpha_0,c)$-protection, say of height $\tau=\height(z)$. Using the same technique as in the proof of Proposition~\ref{lem:2}, point~\ref{item:9}, by adding as many letters as one may while not changing the height of~$z$, we assume without loss of generality that $z$ is a maximal element among those of height~$\tau$. It implies that the path in \DSC\ corresponding to $z$ goes through positive nodes only. Since $\Ttilde$ is a terminal component of~$\DSC^+$, the path corresponding to $z$ in $\DSC$ lies within~$\Ttilde$. Put $\beta=\alpha_0\cdot z$, and let $\bigl((\alpha_0,d_1),\ldots,(\alpha_{\tau-1},d_\tau)\bigr)$ be the path in $\DSC^+$ corresponding to~$z$. Let $(Z_i)_{i\geq0}$ denote the Markov chain of states-and-cliques, with $Z_i=(\alpha_i,C_{i+1})$. Then the maximality of $z$ implies:
\begin{gather*}
  \up z=\Bigl\{\xi\in\M_{\alpha_0}\tqs \bigl(Z_0(\xi),\ldots,Z_{\tau-1}(\xi)\bigr)=
  \bigl((\alpha_0,d_1),\ldots,(\alpha_{\tau-1},d_\tau)\bigr)\Bigr\}.
\end{gather*}

Consider $\alpha\in X$ an arbitrary state, and pick $x\in\M_\beta$ such that $\beta\cdot x=\alpha$. Let $\rho=\height(x)$. Then, for any $\xi\in\BM_{\alpha_0}$, one has $zx\leq\xi$ if and only if the truncature of $\xi$ at height $\tau+\rho$, defined by $Y(\xi)=C_1(\xi)\cdots C_{\tau+\rho}(\xi)$ satisfies $zx\leq Y(\xi)$. Therefore $\up(zx)$ decomposes as the following finite disjoint union:
\begin{gather*}
  \up (zx)=\bigcup_{\substack{y\in\M_{\alpha_0}\tqs\\ (\height(y)=\rho+\tau)\wedge (zx\leq y)}}\bigl\{\xi\in\BM_{\alpha_0}\tqs Y(\xi)=y\bigr\}.
\end{gather*}

But each of the subsets $\{\xi\in\BM_{\alpha_0}\tqs Y(\xi)=y\bigr\}$ is an elementary cylinder for the Markov chain of states-and-cliques. Thanks to the result of~$(\dag)$, their probability evaluates thus identically with respect to $\nu_{\alpha_0}$ and with respect to~$\nu'_{\alpha_0}$, henceforth:\quad$\nu_{\alpha_0}\bigl(\up(zx)\bigr)=\nu'_{\alpha_0}\bigl(\up(zx)\bigr)$, which yields $\Delta(\alpha_0,\alpha)=\Gamma(\alpha_0,\alpha)$. Since the state $\alpha$ was chosen arbitrarily, the cocycle property of $\Gamma$ and of $\Delta$ implies thus $\Gamma=\Delta$, which completes the proof.
\end{proof}

\begin{corollary}
  \label{cor:1}
  Let $(\M,X,\bot)$ be a irreducible concurrent system of characteristic root~$r$, and let $\mu=\mu(r)$ be the M\"obius matrix evaluated at~$r$. Then $\dim\bigl(\ker(\mu)\bigr)=1$.
\end{corollary}

\begin{proof}
  We already know that  $\dim\bigl(\ker(\mu)\bigr)\geq1$, either from Proposition~\ref{prop:9} \emph{via} the existence of the uniform measure, or more directly \emph{via} Proposition~\ref{prop:4} which says that $\det\mu(r)=0$.

  Seeking a contradiction, assume that  $\dim\bigl(\ker(\mu)\bigr)>1$.   Let $\Gamma:X\times X\to \bbR_{>0}$ be the Parry cocycle, and let $f_\alpha(x)=r^{|x|}\Gamma(\alpha,\alpha\cdot x)$ for $\alpha\in X$ and $x\in\M_\alpha$. Fix $\alpha_0$ an arbitrary state and define the vector $u=(u_\alpha)_{\alpha\in X}$ by $u_\alpha=\Gamma(\alpha_0,\alpha)$. According to Proposition~\ref{prop:9}, $u\in\ker\mu$. Let $v$ be a non zero vector of $\ker\mu$, non proportional to~$u$. Choose $\varepsilon>0$ such that $w=u+\varepsilon v>0$, which exists since $u>0$, and let $\Delta:X\times X\to\bbR_{>0}$ be the cocycle defined by:
  \begin{gather*}
\forall (\alpha,\beta)\in X\times X\quad\Delta(\alpha,\beta)=\frac{w_\beta}{w_\alpha}.
  \end{gather*}

  Note that $(\Delta(\alpha_0,\beta))_{\beta\in X}$ is proportional to $u+\varepsilon v$, hence is not proportional to~$u$. In particular, $\Delta\neq\Gamma$.

  Let $f'=(f'_\alpha)_{\alpha\in X}$ be defined by $f'_\alpha(x)=r^{|x|}\Delta(\alpha,\alpha\cdot x)$. We prove that $f_\alpha=f'_\alpha$ for all $\alpha\in X$, which will contradict the previous observation $\Delta\neq\Gamma$.

Let $h'_\alpha$ be the M\"obius transform of~$f'_\alpha$, and let $h_\alpha$ be the M\"obius transform of~$f_\alpha$. We derive the following expression from a straightforward computation, valid for every $\alpha\in X$ and for every $c\in\C_\alpha$:
\begin{gather}
  \label{eq:19}
  h'_\alpha(c)=
  \frac1{w_\alpha}\Big(\Gamma(\alpha_0,\alpha)h_\alpha(c)+\varepsilon\sum_{c'\in\C_\alpha\tqs c'\geq c}
  (-1)^{|c'|-|c|}r^{|c'|}v_{\alpha\cdot c'}\Bigr).
\end{gather}
This yields in particular:
\begin{gather}
  \label{eq:20}
  \forall\alpha\in X\quad h'_\alpha(\vd)=\frac1{w_\alpha}\Bigl(\Gamma(\alpha_0,\alpha)h_\alpha(\vd)+
  \varepsilon(\mu v)_\alpha\Bigr).
\end{gather}

Since $h_\alpha(\vd)=0$ and since $v\in\ker\mu$, we derive $h'_\alpha(\vd)=0$ from~(\ref{eq:20}). This is enough to insure the identity $h'_\alpha=f'_\alpha g'_\alpha$ where $g'_\alpha$ is defined relatively to $h'_\alpha$ as in~(\ref{eq:26}).

For any positive node $(\alpha,c)$, one has $h_\alpha(c)>0$ according to Lemma~\ref{lem:5}. Therefore,  $h'_\alpha(c)>0$ as well according to~(\ref{eq:19}), maybe after having diminished the value of~$\varepsilon$. In particular, following the same lines as in the proof of Lemma~\ref{lem:4}, one obtains from $h'_\alpha$ a Perron eigenvector of the adjacency matrix of some terminal strongly connected component of $\ADSC^+$. Since Perron eigenvectors are unique up to proportionality, we derive that $h'_\alpha(c)=h_\alpha(c)$ for all nodes $(\alpha,c)$ of the corresponding component of $\DSC^+$.  Following now the line of proof of Theorem~\ref{thr:4}, analyzing the Markov chain of states-and-cliques starting from a given node of this terminal component, we obtain that $f_\alpha=f'_\alpha$ for all $\alpha\in X$, yielding the desired contradiction.
\end{proof}

\subsection{Positive and null nodes of $\DSC$ from a probabilistic point of view}
\label{sec:null-nodes-dsc}

Let $(\M,X,\bot)$ be an irreducible concurrent system. We consider the associated uniform measure $\nu=(\nu_\alpha)_{\alpha\in X}$. The aim of the following result is to give an alternative, probabilistic characterization of positive and of null nodes of $\DSC$. The key ingredient in the proof is the spectral property.

\begin{theorem}
  \label{thr:5}
  Let $(\M,X,\bot)$ be an irreducible concurrent system. Let   $f=(f_\alpha)_{\alpha\in X}$ be the fibred valuation induced by the uniform measure. For each state $\alpha\in X$, let $h_\alpha:\C\to\bbR$ be the M\"obius transform of the function $f_\alpha:\C\to\bbR$.  

Let $(\alpha,c)$ be a node of the \DSC. Then the following statements are equivalent:
  \begin{enumerate}
  \item\label{item:13} $(\alpha,c)$ is a positive node.
  \item\label{item:14} $h_\alpha(c)>0$.
  \end{enumerate}
Furthermore, for every $\alpha\in X$, the \MCSC\ under $\nu_\alpha$ only visits positive nodes of the \DSC.
\end{theorem}

\begin{proof}
  $\ref{item:13}\implies\ref{item:14}$.\quad This is the topic of Lemma~\ref{lem:5}.
  $\ref{item:14}\implies\ref{item:13}$.\quad We proceed by contraposition. Let $(\alpha,c)$ be a null node; we prove that $h_\alpha(c)=0$. It follows from Proposition~\ref{lem:2}, point~\ref{item:10}, that:
  \begin{gather*}
    \bigl\{\xi\in\BM_\alpha\tqs C_1(\xi)=c\bigr\}\subseteq \bigcup_{a\in \Sigma}\BM_\alpha^a,
  \end{gather*}
where $\M^a=\langle\Sigma\setminus\{a\}\rangle$ and $\BM^a$ denotes the boundary at infinity of~$\M^a$. Using the property $\nu_\alpha(C_1=c)=h_\alpha(c)$, it is thus enough to prove:
\begin{gather*}
  \forall a\in\Sigma\quad\nu_\alpha\bigl(\BM_\alpha^a\bigr)=0.
\end{gather*}

For this, let $a\in\Sigma$. Let $r$ denote the characteristic root of the concurrent system $(\M,X,\bot)$. For each integer $n\geq0$, one has: $\BM^a\subseteq\bigcup_{x\in\M^a_\alpha(n)}\up x$ and therefore:
\begin{align}
  \label{eq:37}
  \nu_\alpha(\BM^a)\leq\sum_{x\in\M^a_\alpha(n)}\nu_\alpha(\up x)\leq
  K\card\M_\alpha^a(n) r^n
\end{align}
where $K$ is a bound of the Parry cocycle. According to Theorem~\ref{thr:1}, the irreducible concurrent system $(\M,X,\bot)$ satisfies the spectral property. Hence, passing to the limit~(\ref{eq:37}) when $n\to\infty$ yields $\nu_\alpha(\BM^a)=0$, which was to be proved.

The last statement of the theorem now follows from the observation already made that the \MCSC\ only visits nodes such that $h_\alpha(c)>0$.
\end{proof}

\subsection{Further investigations on the uniform measure and on null nodes}
\label{sec:an-alternative-point}

\subsubsection{Finite uniform distributions.}
\label{sec:finite-unif-distr}

Let $(\M,X\bot)$ be a concurrent system, that we assume to be non trivial and irreducible. In particular, for each state $\alpha\in X$ and for each integer $n\geq0$, the finite set $\M_\alpha(n)$ is non empty. Therefore the uniform distribution on $\M_\alpha(n)$ is well defined, let us denote it by~$\nu_\alpha(n)$. A natural question is to elucidate, for each state $\alpha\in X$, whether the sequence $(\nu_\alpha(n))_{n\geq0}$ converges as $n\to\infty$. Indeed, one sees each finite distribution $\nu_\alpha(n)$ as a probability measure on the \emph{compact} space~$\Mbar$.

Finite executions in $\M_\alpha(n)$ correspond to some of the paths of length $n$ in the digraph $\ADSC$, those with some constraints on the initial and final node---see Proposition~\ref{prop:5}. Accordingly, the finite distribution $\nu_\alpha(n)$ corresponds to a finite distribution, say~$\nutilde_\alpha(n)$, on a subset of the paths of length $n$ in~$\ADSC$.

Recall that $\ADSC$ has no reason to be strongly connected in general, nor aperiodic. Studying the convergence of $(\nutilde_\alpha(n))_{n\geq0}$ amounts thus in studying the convergence of finite uniform distributions on general finite graphs. Based on the existing literature on reducible non negative matrices, see \eg\ \cite{rothblum14} and the references therin, it is possible to prove the following result\footnote{Recall that a strongly connected component of a digraph is called \emph{basic} if its own spectral radius equals the spectral radius of the digraph.}, as we shall do in a forthcoming paper: \emph{the sequence $(\nutilde_\alpha(n))_{n\geq0}$ converges weakly as $n\to\infty$ toward a probability measure on infinite paths in $\ADSC$, corresponding to a Markov chain on $\ADSC$. Restricted to the nodes visited with positive probability by this Markov chain, the digraph has the property that its basic components coincide with its terminal components.}

The reduced digraph of the \ADSC---\ie, the digraph of its irreducible components---can thus be represented in a schematic way as in Figure~\ref{fig:pqojwdpaz}. The limit measure corresponds also to the uniform measure on infinite executions introduced earlier in the paper; this justifies \emph{a posteriori} our definition of the uniform measure. The limit Markov chain on \ADSC\ is nothing but the Markov chain of states-and-cliques, and the sub-digraph of $\ADSC$ visited with positive probability is~$\ADSC^+$.

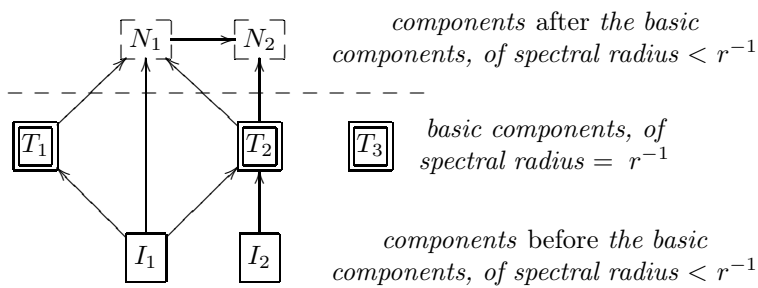
\begin{figure}
$$
  \xymatrix{
    &*+[F--]{N_1}\ar[r]&*+[F--]{N_2}&{\ }&
    &{}\save[]*\txt<18em>{\itshape components \emph{after} the basic components, of spectral radius $<r^{-1}$}\restore\\
    *+[F=]{\strut T_1}\ar[ur]
        \save\POS[]+<-1em,2em>\ar@{--}[rrr]+<2em,2em>
    \restore
&&*+[F=]{\strut T_2}\ar[u]\ar[ul]&*+[F=]{\strut T_3}&
&{}\save[]*\txt<18em>{\itshape basic components, of spectral radius $=r^{-1}$}\restore\\
&    *+[F]{\strut I_1}\ar[ul]\ar[ur]\ar[uu]&*+[F]{\strut I_2}\ar[u]
&&&{}\save[]*\txt<18em>{\itshape components \emph{before} the basic components, of spectral radius $<r^{-1}$}\restore
}
$$
  \caption{Reduced digraph of \ADSC\ and nodes reached by the limit of finite uniform distributions. The Markov chain limit of finite uniform distributions does not cross the long dashed line.}
  \label{fig:pqojwdpaz}
\end{figure}

Incidently, variants of this result in the same algebraic vein also provide a proof of uniqueness of the uniform measure. The identification with the uniform measure studied earlier in the paper provides yet another characterization of null nodes if the system is irreducible; whence the following result, completing Theorem~\ref{thr:5} in the light of the previous result (anounced without proof).

\begin{theorem}
\label{thr:3}
Let $(\M,X,\bot)$ be an irreducible and non trivial concurrent system, and let $(\alpha,c)$ be a node of the \ADSC. Then the folloing properties are equivalent:
\begin{enumerate}[i)]
\item $(\alpha,c)$ is a null node.
\item $h_\alpha(c)=0$.
\item\label{item:11} The irreducible component of\/ $(\alpha,c)$ can be reached from at least one basic component of the \ADSC.
\end{enumerate}
\end{theorem}

\subsubsection{Computability of null nodes.}
\label{sec:comp-null-nodes}

The mere definition of null nodes (Definition~\ref{def:10}) does not provide an obvious operational way of determining them. The characterization of $(\alpha,c)$ being a null node if and only if $h_\alpha(c)=0$ given by Theorem~\ref{thr:5} is more operational. Yet, the straightfoward application of Theorem~\ref{thr:5} to detect null nodes assumes an  exact computation of the characteristic root~$r$ of the concurrent system. For this, one can rely on the use of symbolic computation softwares.  But they are soon overtaken when the size of the system grows. Obviously, and as the expriments reveal, numerical computation softwares can handle systems of larger size, but the latter cannot be used to detect a property of the form $h_\alpha(c)=0$. Indeed, computation errors lead in general to non zero values for $h_\alpha(c)$, even if $(\alpha,c)$ is acutally a null node.

Hence symbolic computation cannot be used in practice for large systems, and numerical computations are intrinsically ruled out for the task we are trying to handle.

It is thus remarkable that the last characterization~{\itshape\ref{item:11})} provided by Theorem~\ref{thr:3} actually puts numerical computations back in the game, as follows. After having determined the spectral radius $r^{-1}$ of \ADSC, with some given precision, it is possible to determine numerically the spectral radius of all other components, so as to be sure whether they have a larger or a smaller spectral radius. By this way, all null nodes are determined. This method however  may fail if the \ADSC\ has several basic components. 

If the system has only one basic component, this can be proved by numerical computations. But if the system has several basic components, numerical computations cannot prove it in general.  This property is thus numerically half decidable. In practice, most of the systems one encounters only have one basic component. But  it would be interesting to find sufficient conditions insuring the uniqueness of the basic component of \ADSC.

\subsection{Examples continued}
\label{sec:examples-continued}

\subsubsection{Direct method}
\label{sec:safe-petri-nets-1}

Let $(\M,X,\bot)$ be the irreducible concurrent system introduced in Figure~\ref{fig:elemeytartas}. We have already obtained that the characteristic root is $r=1/2$. The M\"obius matrix evaluated at~$r$, $\mu=\mu(r)$, is thus:
\begin{gather*}
  \mu=\begin{array}{c}
        \alpha_0\\\alpha_1
      \end{array}
      \begin{pmatrix}
        1/4&-1/4\\
        -1/2&1/2
  \end{pmatrix}
\end{gather*}
with kernel generated by $\left(\begin{smallmatrix}  1\\1 \end{smallmatrix}\right)$. Thanks to Proposition~\ref{prop:9}, we deduce that the Parry cocycle $\Gamma$ is identically equal to~$1$. The values of the corresponding fibred valuation $f_\alpha$ and of the M\"obius transform on cliques are given in Table~\ref{tab:oetyriouas}. The interesting point to notice is the following: of course $h_{\alpha_0}(c)=0$ since $c$ is simply not enabled at state~$\alpha_0$. But, less trivially, although $d$ is enabled at~$\alpha_0$, the computation yields $h_{\alpha_0}(d)=0$, which shows directly that $(\alpha_0,d)$ is a null node by Theorem~\ref{thr:5}, contraposition of $\ref{item:13}\implies\ref{item:14}$. Alternatively, we knew that $h_{\alpha_0}(d)=0$ in advance by Theorem~\ref{thr:5}, contraposition of $\ref{item:14}\implies\ref{item:13}$, since it was easy to verify that $(\alpha_0,d)$ is a null node on Definition~\ref{def:10}.

The key point is to notice the difference between $f_{\alpha_0}(d)=1/2>0$ on the one hand, and $h_{\alpha_0}(d)=0$ on the other hand. The first one means that an infinite execution $\xi$ starting from $\alpha_0$ has probability $1/2$ to carry $d$ within its first clique, \ie, $d\leq C_1(\xi)$; whereas the second one means that this same execution has probability $0$ to have $C_1(\xi)=d$. Indeed, since $b$ and $d$ are concurrent, that would imply that $b$ is never used, hence $\xi=(dd\ldots)$, which meets the intuition of an event of probability~$0$.

\begin{table}
$$  \begin{array}{lcccccc}
      \text{clique}    &a&b&c&d&ad&bd\\
      \hline
      f_{\alpha_0}&r=0.5&r=0.5&0&r=0.5&r^2=0.25&r^2=0.25\\
      h_{\alpha_0}&r-r^2=0.25&r-r^2=0.25&0&\fbox{$r-2r^2=0$}&r^2=0.25&r^2=0.25\\
      f_{\alpha_1}&0&0&r=0.5&r=0.5&0&0\\
      h_{\alpha_1}&0&0&r=0.5&r=0.5&0&0
  \end{array}
  $$
  \caption{\small Fibred valuation $f_\alpha$ and its M\"obius transforms $h_\alpha$ for the example of Figure~\ref{fig:elemeytartas}. The framed entry corresponds to the computation $h_{\alpha_0}(d)=0$, showing that $(\alpha_0,d)$ is a null node.}
  \label{tab:oetyriouas}
\end{table}

The transition matrix of the Markov chain of states-and-cliques on $\DSC^+$ is the following:
\begin{gather*}
  M=
  \begin{array}{c}
    (\alpha_0,a)\\(\alpha_0,b)\\(\alpha_0,ad)\\(\alpha_0,bd)\\(\alpha_1,c)\\(\alpha_1,d)
  \end{array}
  \begin{pmatrix}
    .5&.5&0&0&0&0\\
    0&0&0&0&1&0\\
    .25&.25&.25&.25&0&0\\
    0&0&0&0&.5&.5\\
    .25&.25&.25&.25&0&0\\
    0&0&0&0&.5&.5
  \end{pmatrix}
\end{gather*}


\subsubsection{Alternative method: the tiling example}
\label{sec:tiling-example}

Consider the concurrent system $(\M,X,\bot)$ associated with the tilings of the Aztec diamond of order $2$ described in Section~\ref{sec:tilings}. Instead of directly determining the determinant of the M\"obius matrix, we use the notion of uniform measure to determine the characteristic root of the system, together with the Parry cocycle.

Let $\nu=(\nu_\alpha)_{\alpha\in X}$ denote the uniform measure, and let $f_\alpha(x)=r^{|x|}\Gamma(\alpha,\alpha\cdot x)$ be the induced fibred valuation. We first claim that $f_0(a)=1$. Indeed, if an infinite execution $\xi\in\M_0$ does not satisfy $a\leq\xi$, then $C_1(\xi)=b$ and therefore the Markov chain of states-and-cliques enters the null node $(0,b)$. But this occurs with $\nu_0$-probability~$0$, and therefore $\xi\geq a$ with $\nu_0$-probability~$1$. In other words, $f_0(a)=1$, as claimed. Hence $\Gamma(0,1)=1/r$. For the same reasons, $f_0(b)=1$ and $f_0(ab)=1$, whence $\Gamma(0,2)=1/r$ and $\Gamma(0,3)=1/r^2$. We deduce the values of $\Gamma(3,1)$, $\Gamma(3,2)$ and $\Gamma(3,0)$ by the cocycle property. Collecting the different values, we have thus:
\begin{align*}
  \Gamma(0,1)&=1/r&\Gamma(1,0)&=r&\Gamma(0,2)&=1/r&\Gamma(2,0)&=r\\
  \Gamma(3,1)&=r&\Gamma(3,2)&=r&\Gamma(3,0)&=r^2&\Gamma(1,2)&=1
\end{align*}
For symmetry reasons, one has $\Gamma(3,3')=\Gamma(3',3)$, and with the cocycle property this implies $\Gamma(3,3')=1$. Writing down the identity $h_3(\vd)=0$ (a particular instance of~(\ref{eq:27})) yields:
\begin{gather*}
  1-r\Gamma(3,1)-r\Gamma(3,2)-r\Gamma(3,3')+r^2\Gamma(3,0)=0,
\end{gather*}
and thus $r$ is a root of the polynomial $P(z)=1-z-2z^2+z^4$. This polynomial has a unique root in $(0,1)$, so $r$ is this root: $r\approx0.525 $.

So for instance, if one wants to know what is the law of the first clique of an infinite execution under the uniform measure, when the tiling starts from the state~$1$, the answer is given by $\bigl(h_1(a),h_1(b),h_1(ab)\bigr)$. We already know that $h_1(a)=0$ since $(1,a)$ is a null node, which is recovered through a direct computation, and the two other values are given below.
\begin{align*}
  h_1(a)&=f_1(a)-f_1(ab)=r\Gamma(1,0)-r^2\Gamma(1,2)=0\\
  h_1(b)&=f_1(b)-f_1(ab)=r\Gamma(1,3)-r^2\Gamma(1,2)=1-r^2\approx 0.725\\
  h_1(ab)&=f_1(ab)=r^2\Gamma(1,2)=r^2\approx0.275
\end{align*}

\bibliographystyle{plain}
\bibliography{biblio.bib}

\end{document}